\documentclass[11pt,reqno]{amsart}
\usepackage[top=2.5cm, bottom=4.5cm, left=2.5cm, right=2.5cm]{geometry}
\renewcommand{\baselinestretch}{1.1} 

\usepackage[english]{babel}

\usepackage{amssymb}
\usepackage{amsmath}
\usepackage{amsthm}
\usepackage{amsfonts}
\usepackage{graphicx}
\usepackage{tikz-cd}
\usepackage[colorlinks=true, allcolors=black]{hyperref}
\usepackage{comment}
\usepackage{mathtools}
\usepackage{mathdots}
\usepackage{enumitem}
\usetikzlibrary{decorations.pathmorphing} 

\numberwithin{equation}{section}
\numberwithin{equation}{subsection}

\newtheorem{theorem}{Theorem}[subsection]
\newtheorem{lemma}[theorem]{Lemma}
\newtheorem{prop}[theorem]{Proposition}
\newtheorem{corollary}[theorem]{Corollary}
\newtheorem{conj}[theorem]{Conjecture}
\theoremstyle{definition}
\newtheorem{definition}[theorem]{Definition}
\newtheorem{remark}[theorem]{Remark}
\newtheorem{example}[theorem]{Example}
\newtheorem{notation}[theorem]{Notation}

\newcommand{\cO}{{\mathcal O}}
\newcommand{\pU}{{\partial U}}
\newcommand{\bZ}{{\mathbb Z}}
\newcommand{\bC}{{\mathbb C}}
\newcommand{\bP}{{\mathbb P}}
\newcommand{\bN}{\mathbb{N}}
\newcommand{\bH}{\mathbb{H}}
\newcommand{\ra}{\rightarrow}
\newcommand{\bR}{\mathbb{R}}
\newcommand{\bM}{\mathbb{M}}

\title[Lattice cohomology and the embedded topological type of plane curve germs]{Lattice cohomology and the embedded topological type of plane curve singularities}

\author{Alexander A. Kubasch}
\address{ \scriptsize Alfr\'ed R\'enyi Institute of Math.,
Re\'altanoda utca 13-15, H-1053, Budapest, Hungary \newline
 \hspace*{3mm} ELTE - Univ. of Budapest, Dept. of Geo.,
 P\'azm\'any P\'eter s\'et\'any 1/A, 1117, Budapest, Hungary}
\email{\scriptsize kubasch.alexander@renyi.hu}

\author{Gerg\H{o} Schefler}
\address{\scriptsize Alfr\'ed R\'enyi Institute of Math.,
Re\'altanoda utca 13-15, H-1053, Budapest, Hungary \newline
 \hspace*{3mm} ELTE - Univ. of Budapest, Dept. of Geo.,
 P\'azm\'any P\'eter s\'et\'any 1/A, 1117, Budapest, Hungary}
\email{\scriptsize schefler.gergo@renyi.hu}

\thanks{The authors are partially supported by NKFIH Grant ``\'Elvonal (Frontier)'' KKP 144148. }

\begin{document}
\begin{abstract}
    Analytic lattice cohomology is a new invariant of reduced curve singularities. In the case of plane curves, it is an algebro-geometric analogue of Heegaard Floer Link homology. However, by the rigidity of the analytic structure, lattice cohomology can be naturally defined in higher codimensions as well.
    
    In this paper we show that in the case of irreducible plane curve singularities the lattice cohomology is a complete embedded topological invariant. We also compare it to the integral Seifert form in the case of multiple branches.
\end{abstract}
\maketitle

\section{Introduction}

\subsection{Context} \emph{Lattice cohomology theory} encompasses a family of closely related invariants associated to isolated singularities of complex analytic varieties. The \emph{topological lattice cohomology of normal surface singularities} was introduced in \cite{N08} in an attempt to relate their analytic invariants with the topology of their link, most prominently, its Heegaard Floer homology. Since then, building from similar ideas, many other versions have been constructed including the \emph{analytic lattice cohomology of normal surface singularities} \cite{AgNeSurf} and the \emph{analytic lattice cohomology of isolated curve singularities} \cite{AgNeCurve}. In the present paper we compare this latter invariant to the embedded topological type and the Seifert form of plane curve singularities.

The topological lattice cohomology of a normal surface singularity with $\mathbb{Q}HS^3$ link was recently proved to be isomorphic to the Heegard-Floer (HF) homology of its link \cite{Z21}. Similarly, the analytic lattice cohomology of a plane curve singularity is related to Heegard Floer Link (HFL) homology via a spectral sequence \cite{FiltLC}. The diagram below outlines some of these connections.\vspace{1mm}
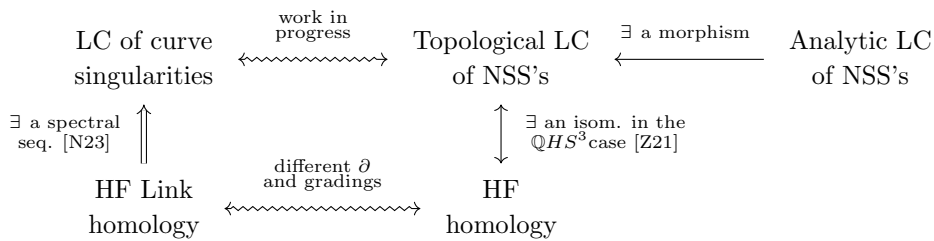
\begin{center}
\small
        \begin{tikzcd}[ampersand replacement=\&, column sep = 20mm, row sep = 8mm]
            \begin{array}{c} \text{LC of curve} \\ \text{singularities} \end{array} \arrow[r, leftrightsquigarrow,"\substack{\text{work in}\\ \text{progress} \\ \quad}"] \arrow[d, Leftarrow, "\substack{\exists \ \text{a spectral} \\ \text{seq. \cite{FiltLC}}} \ \ "'] \& \begin{array}{c} \text{Topological LC} \\ \text{of NSS's} \end{array} \arrow[d, leftrightarrow, "\ \ \substack{\exists \text{ an isom. in the}\\ \mathbb{Q}HS^3 \text{case \cite{Z21}}}"] \& \begin{array}{c} \text{Analytic LC} \\ \text{of NSS's}\end{array} \arrow[l,swap, "\ \substack{\exists \ \text{a morphism} \\ \quad} \ \ "]\\
            \begin{array}{c} \text{HF Link} \\ \text{homology} \end{array} \arrow[r, swap, leftrightsquigarrow, "\substack{ \quad \\ \quad  \\ \text{different } \partial \\ \text{and gradings} \\ \quad}"'] \& \begin{array}{c} \text{HF} \\ \text{homology} \end{array} \&
        \end{tikzcd}
    \end{center}\vspace{1mm}
    Thus, one can think of lattice cohomology as an complex analytic analogue of Heegaard-Floer theory. Note however, that despite of these deep connections to low-dimensional topology, \emph{all} lattice cohomology theories have purely analytic/algebro-geometric constructions.

    \subsection{Curve singularities} The study of curve singularities, particularly in the complex plane, is a crucial area within algebraic geometry. Initially driven by algebraic and analytic tools, its understanding has been enriched by low-dimensional topology. This interdisciplinary approach has yielded a wealth of algebraic and topological invariants and deep insights. More recently invariants like the HOMFLY polynomial, HF homology and Khovanov homology have further propelled this field. This paper fits into this context by studying the lattice cohomology of curve singularities, a complex analytic analogue of the HF theory which works in higher codimensions as well.
    
    \subsection{Basic structure of lattice cohomology} The lattice cohomology module of an isolated curve singularity $(C,o)$ is a bigraded $\bZ[U]$-module denoted by $\bH^*(C,o)$. It admits a \emph{cohomological grading} $\bH^*(C,o) = \bigoplus_q \bH^q(C,o)$ and a \emph{weight grading} $\bH^q(C,o) = \bigoplus_{2n}\bH^q_{2n}(C,o)$. This latter grading is \emph{even} in order to emphasize the compatibility with the HFL theory. If $(C,o) = \bigcup_{i=1}^r(C_i,o)$ has $r$ irreducible components, then $\bH^q(C,o)=0$ for $q \geq r$. Moreover, $\bH^*_{2n}(C,o) = 0$ for $n \ll 0$ \cite{AgNeCurve} and $\bH^*_{2n}(C,o) = \bH^0_{2n}(C,o) \cong \bZ$ for $n > 0$ \cite{StrucProp}.  The Euler characteristic of $\bH^*(C,o)$ is the delta invariant $\delta(C,o)$ \cite{AgNeCurve} and $\bH^0(C,o)$ determines wether or not $(C,o)$ is Gorenstein \cite{StrucProp}.
    
    \subsection{Embedded topological type and other invariants}

    For planar germs, lattice cohomology is an embedded topological invariant and in this case $\bH^0(C,o)$ determines the multiplicity $m(C,o)$ \cite{LCandmult,StrucProp}. It is natural to ask how $\bH^*(C,o)$ is related to other topological invariants.

    \begin{theorem}\normalfont \textbf{[Theorem \ref{thm:MAIN}]} \textit{Let $(C,o)$ be an irreducible plane curve singularity. Then $\bH^*(C,o) = \bH^0(C,o)$ is a complete embedded topological invariant.}        
    \end{theorem}

    \noindent Even more, in subsection \ref{subs:alg} we give an explicit algorithm for reconstructing the semigroup $\mathcal{S}_{C,o}$ from the graded $\bZ[U]$-module structure of $\bH^*(C,o)$. The proof of this statement (together with the necessary preliminaries) takes up the majority of this article. It is based on the technique of `local minima' developed in \cite{LCandmult, StrucProp}.

    We also present examples of plane curve singularities $(C,o)$ with multiple branches showing that $\bH^*(C,o)$ and the Seifert form $S(C,o)$ do not determine each other. It follows, in particular, that in this case the lattice cohomology module is not a complete embedded topological invariant. Note however that there exists an enhanced version of this theory, the \emph{filtered lattice homology of curve singularities}, which \emph{is} a complete topological invariant even in this case \cite{FiltLC}.

    The multivariate Alexander polynomial $\Delta_{C,o}(t_1,\dots,t_r)$ is a complete embedded topological invariant of plane curve singularities \cite{Y84}. In particular, it determines $\bH^*(C,o)$. In the final section of this paper we present an example that this is no longer the case for non-planar curves (where the Alexander polynomial is replaced by the multivariate Poincaré series).

\subsection{Structure of the paper.} In section 2 we give a brief overview of the analytic lattice cohomology of curve singularities. Section 3 is devoted to proving the main theorem. We introduce the required machinery and give several examples. In section 4 we compare the analytic lattice cohomology with the Seifert form of plane curve singularities. In section 5 we present its relation with the multivariable Poincaré series in the non-planar case.

\subsection{Acknowledgements.} The authors would like to express their deep gratitude to Professor András Némethi for his constant help and encouragement.

\section{Lattice cohomology of curve singularities}

\subsection{} In this section we give a \emph{very} brief description of the analytic lattice cohomology of curve singularities introduced in \cite{AgNeCurve}. For more details on the theory see \cite{AgNeCurve, FiltLC, LCandmult, StrucProp}.\vspace{3mm}

\noindent Let $(C,o)$ be a reduced complex curve singularity and $(C,o) = \bigcup_{i=1}^r(C_i,o)$ its irreducible decomposition . Denote the local algebras by $\mathcal{O}:=\mathcal{O}_{C,o}$ and $\mathcal{O}_i := \mathcal{O}_{C_i,o}$ respectively. The integral closure of $\mathcal{O}$ is $\overline{\mathcal{O}} = \bigoplus_i \overline{\mathcal{O}_i} \cong \bigoplus_i \bC\{t_i\}$ and the inclusions (i.e. the normalization maps) $n_i^* : \mathcal{O}_i \hookrightarrow \overline{\mathcal{O}_i}$ bring about the valuations
$$\mathfrak{v}_i : \mathcal{O} \to \mathcal{O}_i \to \bN \cup \{\infty \} \ \ ; \quad f \mapsto \text{ord}_{t_i}(n_i^*f).$$
They induce the $\bZ^r$-filtration $\mathcal{F}$ on $\mathcal{O}$ given by $\mathcal{F}(\ell) = \{\, f \in \mathcal{O}\ | \ \mathfrak{v}_i(f) \geq \ell_i \text{ for all }i \, \}$. Let $\mathfrak{h} : \bZ^r \to \bN$ denote the corresponding Hilbert function defined as $\mathfrak{h}(\ell) := \dim_\bC \mathcal{O} \big/ \mathcal{F}(\ell)$.

\begin{remark}\label{rem:poincaré}
    It is often convenient to encode the Hilbert function $\mathfrak{h}$ in the formal \textit{Hilbert series} $H_{C,o}(t) := \sum_{\ell \in \bZ^r}\mathfrak{h}(\ell)t^\ell$. The \textit{Poincaré series} is defined as $P_{C,o}(t):=-H(t) \cdot \prod_{i=1}^r\big(1-\frac{1}{t_i}\big)$. The former always determines the latter but the converse does not always hold, see section \ref{sec:poincaré}. 
\end{remark}

\begin{definition}
    The \textit{weight function} $w_0 : \bN^r \to \bZ$ is defined as
    $$w_0(\ell) := 2\cdot \mathfrak{h}(\ell) - |\ell |$$
    where $|\ell | = \sum_i\ell_i$. The \textit{extended weight function} $w : \bR^r_{\geq 0} \to \bZ$ is defined as
    $$w(x) = \max\big\{\, w_0(\ell) \ \big| \ \ell_i = \lfloor x_i \rfloor \text{ or } \lceil x_i \rceil \text{ for all }i \,\big\}.$$
\end{definition}
\noindent Equivalently, the lattice $\bN^r \subset \bR^r_{\geq 0}$ induces a cubical decomposition of $\bR^r_{\geq 0}$ and $w(x)$ is the maximal weight attained by the vertices of the smallest cube containing $x$.

Given an integer $n$, define the sublevel set
$S_n := \{\, x \in \bR^r_{\geq 0} \ | \ w(x) \leq n \,\}.$ It turns out that $w_0$ is bounded from below, see e.g. \cite{AgNeCurve} or Proposition \ref{rem:prelim}. Furthermore, $S_n$ is a finite cubical complex for all $n \geq \min w_0$ and empty for $n < \min w_0$.

\begin{definition}
    The lattice cohomology module of $(C,o)$ is defined as
    \begin{equation*}
        \bH^*(C,o) := \bigoplus_{q=0}^{r-1}\bigoplus_{n \in \bZ} \bH^q_{2n}(C,o), \text{ where }\bH^q_{2n}(C,o) := H^q(S_n,\bZ).
    \end{equation*}
    It admits the structure of a $\bZ[U]$-module with $U$-action $\bH^q_{2n+2}(C,o) \to \bH^q_{2n}(C,o)$ induced by the inclusions $S_n \hookrightarrow S_{n+1}$.
\end{definition}
\begin{theorem} \label{thm:prop} The following list provides a short summary of some of the more important properties of this invariant.

\begin{enumerate}
    \item The Euler characteristic of $\bH^*(C,o)$ is the delta invariant $\delta(C,o) :=\dim_\mathbb{C} \overline{\mathcal{O}}/\mathcal{O}$, see \cite[Corollary 4.2.2]{AgNeCurve}.
    \item The $\mathbb{Z}[U]$-submodule $\bH^0(C,o)$ detects whether or not $(C,o)$ is Gorenstein, see \cite[Theorem 4.2.1]{StrucProp}
    \item In the case of a plane curve singularity $\bH^*(C,o)$ determines the multiplicity $m(C,o)$, see \cite{LCandmult} for the irreducible case and \cite[Section 5]{StrucProp} for the general case.
    \item
    The space $S_n$ is contractible for $n> 0$, hence $\bH^*_{2n}(C,o) = \bH^0_{2n}(C,o) \cong \bZ$ for all $n > 0$, see \cite[Nonpositivity Theorem 6.1.2]{StrucProp}.
\end{enumerate}
\end{theorem}

\begin{notation} \label{def:ZUmod}
    Given two integers $m \leq n$ with $n$ possibly being equal to $\infty$, define the graded $\bZ[U]$-module $T_{2m}^{2n}$ as follows. Consider the graded Abelian group $A=A_{2n}\oplus A_{2n-2}\oplus \dots \oplus A_{2m}$ where $A_{2j} \cong \bZ$ for all $m \leq j \leq n$. Endow it with the $\bZ[U]$-module structure given by isomorphisms $U : A_{2j} \xrightarrow{\, \cong \,} A_{2j-2}$ if $m < j \leq n$ and let $UA_{2m}=0$. In the case $m=n$, we will also use the notation $T_{2m} := T_{2m}^{2m}$
\end{notation}

\begin{prop}\cite[Remark 2.7]{LCandmult} \label{prop:directsumdecomp}
    The graded $\bZ[U]$-module $\bH^0(C,o)$ admits a direct sum decomposition
    $\bH^0(C,o) \cong T_{2\min w_0}^\infty\oplus \bigoplus_k T_{2m_k}^{2n_k}$
    which is unique (but not canonical) up to graded $\bZ[U]$-module isomorphism.
\end{prop}

\noindent Note also the following corollary of the Nonpositivity Theorem (item \textit{(4)} of Theorem \ref{thm:prop}).

\begin{corollary}\label{cor:canemb}
    In fact, the graded $\bZ[U]$-submodule $T^\infty_{2 \min w_0}$ mentioned in Proposition \ref{prop:directsumdecomp} is canonically embedded into $\bH^0(C,o)$. Indeed, $\bH_{2n}^0\cong \bZ$ and $U : \bH^0_{2n+2} \xrightarrow{\, \cong \, } \bH^0_{2n}$ for any $n>0$, and  $T^\infty_{2 \min w_0} \subseteq \bH^0$ is the $\bZ[U]$-submodule generated by $\bH^0_{> 0} := \bigoplus_{n>0}\bH^0_{2n}$.
\end{corollary}

\section{The lattice cohomology of plane branches}
\subsection{} The main result of this section is the following

\begin{theorem}\label{thm:MAIN}
    Let $(C,o) \subset (\bC^2,o)$ be an irreducible plane curve singularity. Then the graded $\bZ[U]$-module $\bH^0(C,o)$ is a complete embedded topological invariant.
\end{theorem}

\begin{remark}
    Note that for irreducible curve singularities $\bH^{q}(C,o) = 0$ for all $q \geq 1$ by construction. (This is not true in general: see subsection \ref{subs:seifert} for an example with $\bH^1 \neq 0.)$
\end{remark}

\noindent Before proving Theorem \ref{thm:MAIN}, we will need some preliminaries.

\subsection{Preliminaries on irreducible plane curve singularities}

\begin{prop} \label{rem:prelim}
 Let $(C,o) \subset (\bC^2,o)$ be an irreducible plane curve singularity. The semigroup of values of $(C,o)$ is defined as
    $$\mathcal{S}_{C,o}:= \{\, \text{\normalfont ord}(n^*f) \ | \ f \in \mathcal{O}_{\bC^2,o} \, \} \cap \mathbb{N},$$
    where $n : (\bC,0) \to (C,o)$ denotes the normalization / Puiseux parametrization (see e.g. \cite{Garcia, delaMata87} for more details). We now list some key properties of $\mathcal{S}_{C,o}$.

\begin{enumerate}[label=(\roman*)]

    \item The semigroup $\mathcal{S}_{C,o}$ is a complete embedded topological invariant of the curve singularity $(C,o)$, see e.g. \cite{BK86}. 

    \item The semigroup $\mathcal{S}_{C,o}$ is a \emph{numerical semigroup}, meaning that $\#(\bN \setminus \mathcal{S}_{C,o}) < \infty$.
    
    \item The multiplicity of the germ $(C,o)$ is defined as $m(C,o):= \min(\mathcal{S}_{C,o}\setminus \{0\})$ and its delta invariant as $\delta(C,o) := \#(\bN\setminus\mathcal{S}_{C,o})$.
    
    \item The semigroup $\mathcal{S}_{C,o}$ is symmetric in the sense that $\ell \in \mathcal{S}$ if and only if $c-1-\ell \not\in\mathcal{S}$ (see \cite{K70}) where $c:= \min\{ \,\ell \in \bN \ | \ \ell+\bN \subseteq \mathcal{S} \,\}$ is the \textit{conductor}. (Note that $c=2\delta$.)
    
    \item Let $\beta_0,\dots,\beta_g$ be the (unique) minimal generating set of $\mathcal{S}_{C,o}$. (Note that $\beta_0 = m(C,o)$.) Define $l_i := \gcd(\beta_0, \dots, \beta_i)$, and $n_i := l_{i-1}/l_i$. Then
    \begin{equation}\label{eq:nibetai<}
        n_i\beta_i < \beta_{i+1} \text{ for all } 0 \leq i \leq g-1,
    \end{equation}
    see e.g. \cite[Theorem 2]{B72}.
    
\end{enumerate}

    \noindent The next three properties describe the relationship of the semigroup of values $\mathcal{S}_{C,o}$ to the lattice cohomology module $\bH^*(C,o)$.

\begin{enumerate}[resume, label=(\roman*)]
    \item The weight function of $(C,o)$ can be given by the formula 
    \begin{equation} \label{eq:wfromS}
        w_0(\ell) = \#\big([0,\ell)\cap \mathcal{S}_{C,o}\big) - \#\big([0,\ell)\setminus \mathcal{S}_{C,o}\big),
    \end{equation}
    see {\normalfont(\cite{AgNeCurve}, Example 4.7.3 or \cite{LCandmult}, subsection 5.2)}. It immediately follows from item \textit{(iv)} that $w_0$ is symmetric in the sense that 
    \begin{equation} \label{eq:Gorsym}
        w_0(\ell) = w_0(c-\ell) \text{ for all } \ell \in \bN.
    \end{equation} Also, for any $\ell \in \mathbb{N}$ we have 
    \begin{equation}\label{eq:w_0valtozasa}
       \vert w_0(\ell+1)-w_0(\ell)\vert\equiv1 \text{ as } w_0(\ell+1) = \begin{cases}
            w_0(\ell)+1, \hspace{5mm} \text{if }\ell \in \mathcal{S}_{C,o};
            \\
            w_0(\ell)-1, \hspace{5mm} \text{if }\ell \notin \mathcal{S}_{C,o}.
        \end{cases}
    \end{equation}
    \item The multiplicity $m(C,o)$ can be determined from $\bH^0(C,o)$ as follows (see \cite{LCandmult}). First, $m(C,o) \leq 2$ if and only if $\min\{\,n \in \bZ \ | \ \bH^0_{2n}(C,o) \not= 0 \,\}=0$. In this case the germ $(C,o)$ is smooth if and only if $ {\rm rank}_{\bZ}\bH^0_0(C,o) = 1$ and $m(C,o)=2 $ otherwise. On the other hand $m(C,o)>2$ if and only if $\min\{\,n \in \bZ \ | \ \bH^0_{2n}(C,o) \not= 0 \,\}<0$ and in this case
    $$m(C,o) = 2-\max\big\{\, n < 0 \ | \ \ker\big(U:\bH^0_{2n}(C,o)\to\bH^0_{2n-2}(C,o)\big) \neq 0 \, \big\}.$$
    \item The delta invariant can be determined from $\bH^0(C,o)$ via the formula
    $$\delta(C,o) = \text{\normalfont rank}_{\bZ}\bH^0_{\leq 0}(C,o)-1,$$ see \cite[Remark 2.3.3 (c)]{StrucProp}.
\end{enumerate}
\end{prop}

 \noindent We remark that statements \textit{(i)-(viii)} of Proposition \ref{rem:prelim} all have analogues in the case of plane curves with multiple branches.

 \begin{corollary} \label{cor:w(s)leq0}
    For every $s \in \mathcal{S}_{C,o} \cap [0, c]$, we have $w(s) \leq 0$.
\end{corollary}

\begin{proof}
    Let us suppose indirectly, that $w_0(s)\geq1$. From Gorenstein symmetry (\ref{eq:Gorsym}) we get $w_0(0)=w_0(c)=0$, hence $s \in \mathcal{S}_{C, o} \cap [1, c-1]$. Even more, from equation (\ref{eq:w_0valtozasa}) we see that $w_0(s+1)=w_0(s)+1 >1$, with $s+1 \in [2, c-2]$. But this implies that $0$ and $c$ are contained in different connected components of $S_1$, which contradicts its  contractibility.
\end{proof}

\subsection{Graded roots}

\begin{definition}\cite[3.2 (a)]{NOSz}\label{def:gradedroot}
    Let $R$ be an infinite tree with vertex set $V$ and edge set $E$ equipped with a grading $\chi : V \to \bZ$. We say that the pair $(R,\chi)$ is a \emph{graded root} if
    \begin{itemize}
        \item[(a)] $\chi(u)-\chi(v) = \pm 1$ for eny edge $\{u,v\} \in E$
        \item[(b)] $\chi(u) > \min\big\{\chi(v),\chi(w)\big\}$ whenever $\{u,v\},\{u,w\} \in E$
        \item[(c)] $\chi$ is bounded from below, $\big|\,\chi^{-1}(k)\,\big|$ is finite for all $k$ and equal to $1$ for all $k \gg 0$.
    \end{itemize}
\end{definition}

\begin{definition}
    Let $(R,\chi)$ be a graded root. We associate to it a graded $\bZ[U]$-module $\bH(R,\chi)$. Its underlying Abelian group is the free Abelian group $\bZ \langle V \rangle$ generated by the vertex set $V$ of $R$. The grading of the $\bZ$-summand $v\cdot \bZ \leq \bZ\langle V\rangle$ generated by $v \in V$ is set to be $\chi(v)$. Given any $v \in V$, the $U$-action is defined as
    $$U\cdot v := \sum_{\substack{ w \in V \\ \{v,w\} \in E \\ \chi(w) = \chi(v)-1 }}  w.$$
\end{definition}

\begin{example}\label{eg:rootlattice} Consider the two graded roots depicted in the diagram below.
    \begin{center}
        
\tiny

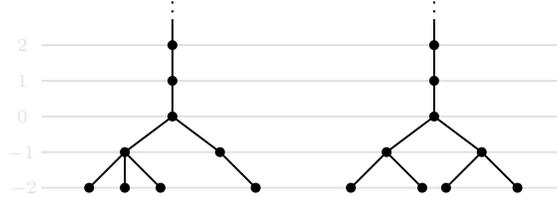
\tikzset{every picture/.style={line width=0.75pt}} 

\begin{tikzpicture}[x=0.75pt,y=0.75pt,yscale=-.6,xscale=.6]

\draw [color={rgb, 255:red, 225; green, 225; blue, 225 }  ,draw opacity=1 ]   (70,60) -- (510,60) ;
\draw [color={rgb, 255:red, 225; green, 225; blue, 225 }  ,draw opacity=1 ]   (70,90) -- (510,90) ;
\draw [color={rgb, 255:red, 225; green, 225; blue, 225 }  ,draw opacity=1 ]   (70,120) -- (510,120) ;
\draw [color={rgb, 255:red, 225; green, 225; blue, 225 }  ,draw opacity=1 ]   (70,150) -- (510,150) ;
\draw [color={rgb, 255:red, 225; green, 225; blue, 225 }  ,draw opacity=1 ]   (70,180) -- (510,180) ;
\draw    (140,150) -- (110,180) ;
\draw [shift={(110,180)}, rotate = 135] [color={rgb, 255:red, 0; green, 0; blue, 0 }  ][fill={rgb, 255:red, 0; green, 0; blue, 0 }  ][line width=0.75]      (0, 0) circle [x radius= 3.35, y radius= 3.35]   ;
\draw    (140,150) -- (170,180) ;
\draw [shift={(170,180)}, rotate = 45] [color={rgb, 255:red, 0; green, 0; blue, 0 }  ][fill={rgb, 255:red, 0; green, 0; blue, 0 }  ][line width=0.75]      (0, 0) circle [x radius= 3.35, y radius= 3.35]   ;
\draw    (180,120) -- (140,150) ;
\draw [shift={(140,150)}, rotate = 143.13] [color={rgb, 255:red, 0; green, 0; blue, 0 }  ][fill={rgb, 255:red, 0; green, 0; blue, 0 }  ][line width=0.75]      (0, 0) circle [x radius= 3.35, y radius= 3.35]   ;
\draw    (140,150) -- (140,180) ;
\draw [shift={(140,180)}, rotate = 90] [color={rgb, 255:red, 0; green, 0; blue, 0 }  ][fill={rgb, 255:red, 0; green, 0; blue, 0 }  ][line width=0.75]      (0, 0) circle [x radius= 3.35, y radius= 3.35]   ;
\draw    (180,120) -- (220,150) ;
\draw [shift={(220,150)}, rotate = 36.87] [color={rgb, 255:red, 0; green, 0; blue, 0 }  ][fill={rgb, 255:red, 0; green, 0; blue, 0 }  ][line width=0.75]      (0, 0) circle [x radius= 3.35, y radius= 3.35]   ;
\draw    (180,90) -- (180,120) ;
\draw [shift={(180,120)}, rotate = 90] [color={rgb, 255:red, 0; green, 0; blue, 0 }  ][fill={rgb, 255:red, 0; green, 0; blue, 0 }  ][line width=0.75]      (0, 0) circle [x radius= 3.35, y radius= 3.35]   ;
\draw    (180,60) -- (180,90) ;
\draw [shift={(180,90)}, rotate = 90] [color={rgb, 255:red, 0; green, 0; blue, 0 }  ][fill={rgb, 255:red, 0; green, 0; blue, 0 }  ][line width=0.75]      (0, 0) circle [x radius= 3.35, y radius= 3.35]   ;
\draw    (180,40) -- (180,60) ;
\draw [shift={(180,60)}, rotate = 90] [color={rgb, 255:red, 0; green, 0; blue, 0 }  ][fill={rgb, 255:red, 0; green, 0; blue, 0 }  ][line width=0.75]      (0, 0) circle [x radius= 3.35, y radius= 3.35]   ;
\draw  [dash pattern={on 0.84pt off 2.51pt}]  (180,40) -- (180,20) ;
\draw    (220,150) -- (250,180) ;
\draw [shift={(250,180)}, rotate = 45] [color={rgb, 255:red, 0; green, 0; blue, 0 }  ][fill={rgb, 255:red, 0; green, 0; blue, 0 }  ][line width=0.75]      (0, 0) circle [x radius= 3.35, y radius= 3.35]   ;
\draw    (360,150) -- (330,180) ;
\draw [shift={(330,180)}, rotate = 135] [color={rgb, 255:red, 0; green, 0; blue, 0 }  ][fill={rgb, 255:red, 0; green, 0; blue, 0 }  ][line width=0.75]      (0, 0) circle [x radius= 3.35, y radius= 3.35]   ;
\draw    (440,150) -- (410,180) ;
\draw [shift={(410,180)}, rotate = 135] [color={rgb, 255:red, 0; green, 0; blue, 0 }  ][fill={rgb, 255:red, 0; green, 0; blue, 0 }  ][line width=0.75]      (0, 0) circle [x radius= 3.35, y radius= 3.35]   ;
\draw    (400,120) -- (360,150) ;
\draw [shift={(360,150)}, rotate = 143.13] [color={rgb, 255:red, 0; green, 0; blue, 0 }  ][fill={rgb, 255:red, 0; green, 0; blue, 0 }  ][line width=0.75]      (0, 0) circle [x radius= 3.35, y radius= 3.35]   ;
\draw    (360,150) -- (390,180) ;
\draw [shift={(390,180)}, rotate = 45] [color={rgb, 255:red, 0; green, 0; blue, 0 }  ][fill={rgb, 255:red, 0; green, 0; blue, 0 }  ][line width=0.75]      (0, 0) circle [x radius= 3.35, y radius= 3.35]   ;
\draw    (400,120) -- (440,150) ;
\draw [shift={(440,150)}, rotate = 36.87] [color={rgb, 255:red, 0; green, 0; blue, 0 }  ][fill={rgb, 255:red, 0; green, 0; blue, 0 }  ][line width=0.75]      (0, 0) circle [x radius= 3.35, y radius= 3.35]   ;
\draw    (400,90) -- (400,120) ;
\draw [shift={(400,120)}, rotate = 90] [color={rgb, 255:red, 0; green, 0; blue, 0 }  ][fill={rgb, 255:red, 0; green, 0; blue, 0 }  ][line width=0.75]      (0, 0) circle [x radius= 3.35, y radius= 3.35]   ;
\draw    (400,60) -- (400,90) ;
\draw [shift={(400,90)}, rotate = 90] [color={rgb, 255:red, 0; green, 0; blue, 0 }  ][fill={rgb, 255:red, 0; green, 0; blue, 0 }  ][line width=0.75]      (0, 0) circle [x radius= 3.35, y radius= 3.35]   ;
\draw    (400,40) -- (400,60) ;
\draw [shift={(400,60)}, rotate = 90] [color={rgb, 255:red, 0; green, 0; blue, 0 }  ][fill={rgb, 255:red, 0; green, 0; blue, 0 }  ][line width=0.75]      (0, 0) circle [x radius= 3.35, y radius= 3.35]   ;
\draw  [dash pattern={on 0.84pt off 2.51pt}]  (400,40) -- (400,20) ;
\draw    (440,150) -- (470,180) ;
\draw [shift={(470,180)}, rotate = 45] [color={rgb, 255:red, 0; green, 0; blue, 0 }  ][fill={rgb, 255:red, 0; green, 0; blue, 0 }  ][line width=0.75]      (0, 0) circle [x radius= 3.35, y radius= 3.35]   ;

\draw (41,172.4) node [anchor=north west][inner sep=0.75pt]  [color={rgb, 255:red, 225; green, 225; blue, 225 }  ,opacity=1 ]  {$-2$};
\draw (39,142.4) node [anchor=north west][inner sep=0.75pt]  [color={rgb, 255:red, 225; green, 225; blue, 225 }  ,opacity=1 ]  {$-1$};
\draw (47,112.4) node [anchor=north west][inner sep=0.75pt]  [color={rgb, 255:red, 225; green, 225; blue, 225 }  ,opacity=1 ]  {$0$};
\draw (47,82.4) node [anchor=north west][inner sep=0.75pt]  [color={rgb, 255:red, 225; green, 225; blue, 225 }  ,opacity=1 ]  {$1$};
\draw (47,52.4) node [anchor=north west][inner sep=0.75pt]  [color={rgb, 255:red, 225; green, 225; blue, 225 }  ,opacity=1 ]  {$2$};

\end{tikzpicture}

    \end{center}
    They are clearly non-isomorphic, however, the associated $\bZ[U]$-modules  are both isomorphic to $T^\infty_{-4} \oplus T^{-2}_{-4} \oplus T_{-4}\oplus T_{-4}.$ It follows that in general $(R,\chi)$ is strictly stronger than $\bH(R,\chi)$
\end{example}

\begin{definition}\label{def:singroot}
Let $(C,o)$ be an isolated curve singularity. Its graded root $\big(R(C,o),\chi\big)$ is defined as follows: the vertices $\{v_{n,i}\}_i$ of weight $\chi = n$ correspond bijectively to the connected components $S_{n,i}$ of $S_n$; and two vertices $v_{n,i}$ and $v_{n+1,j}$ are connected by an edge, if and only if $S_{n,i}\subseteq S_{n+1,j}$.
\end{definition}

\begin{prop}
    Let $(C,o)$ be an isolated curve singularity. Then there is a canonical isomorphism
    $$\bH^0(C,o) \cong \bH\big( R(C,o),\chi \big),$$
    see \cite[Proposition 2.2.6]{AgNeCurve}.
\end{prop}

\noindent According to Theorem \ref{thm:MAIN}, in the case of an irreducible plane curve singularity $(C,o)$, the lattice cohomology module $\bH^0(C,o)$ determines the embedded topological type and hence the graded root $\big(R(C,o),\chi\big)$ as well. In spite of Example \ref{eg:rootlattice}  and that there is no immediate reason to expect this in the case of arbitrary curve singularities, no counterexamples are known. Hence, we propose the following: 

\begin{conj}
    Let $(C,o)$ ben isolated curve singularity. Its graded root $\big(R(C,o),\chi\big)$ is determined by the lattice cohomology module $\bH^0(C,o)$.
\end{conj}

\subsection{Local minima}

Let us fix an irreducible plane curve singularity $(C,o) \subset (\mathbb{C}^2, o)$ with semigroup of values $\mathcal{S}_{C,o} \subset \bZ_{\geq 0}$ and weight function $w_0: \bZ_{\geq 0} \to \bZ$. The authors and A. Némethi defined the local minimum points of the weight function $w_0$ in \cite{LCandmult, StrucProp} as follows:

\begin{definition}\label{def:locmin}
    A lattice point $\ell \in \bZ_{\geq 0}$ is said to be a \textit{local minimum point} of the weight function $w_0$ if either $\ell = 0$, or  $w_0(\ell-1) > w_0(\ell) < w_0(\ell +1)$. (Note that  $w_0(0) =0 < w_0(1)=1$ always holds.)
    Local minimum points are precisely the elements of the set $\mathcal{S}_{C,o} \cap \big( (\bZ \setminus \mathcal{S}_{C,o})+1\big)$ by formula (\ref{eq:w_0valtozasa}). The weight $w_0(\ell)$ of a local minimum point $\ell$ is said to be a \textit{local minimum value}.
\end{definition}

\noindent Let us compare this notion with the local minimum points of the corresponding graded root.

\begin{definition} \cite[3.2 (b)]{NOSz}
    Let $R(C,o)$ be the graded root of the curve singularity $(C,o)$. The vertex $v \in V$ is said to be a local minimum point of the graded root if $\chi(v) < \chi(w)$ for all vertices $w$ adjacent to $v$. Note that, by item (b) of Definition \ref{def:gradedroot}, this is equivalent to $v$ being a leaf (i.e. a vertex with only a single neighbour) of the tree $R(C,o)$.
\end{definition}

\noindent Notice that, by its definition, a local minimum point $\ell \in \bZ_{\geq 0}$ of the weight function $w_0$ with $w_0(\ell)=n$ gives a distinct connected component $\{\ell\}$ of $S_n$. Also, $\{\ell\} \cap S_{n-1} = \emptyset$, hence the vertex corresponding to this connected component is a local minimum point with degree $\chi=n$ of the graded root $R(C, o)$. By formula  (\ref{eq:w_0valtozasa}), all local minimum points of $R(C, o)$ arise in this way. Thus we have the following bijective correspondence \cite{LCandmult,StrucProp}
\begin{equation*}
    \left\{ \begin{array}{c}
        \text{local minimum points of } w_0\\
         \text{ with weight } w_0=n
    \end{array}
\right\}
\overset{ 1-1}{\longleftrightarrow}
\left\{ \begin{array}{c}
        \text{local minimum points of } R(C, o)\\
         \text{ with degree } \chi=n
    \end{array}
\right\}.
\end{equation*}

\noindent It is also clear, that the generator of $\mathbb{H}^0_{2n}$ corresponding to the component $\{\ell\}$ of $S_n$ lives in the kernel of the $U$-action, whence we have the equalities
\begin{equation}\label{eq:locmin}
   \# \left\{ \begin{array}{c}
        \text{loc. min. pts of}\\
         w_0 \text{ with } w_0=n
    \end{array}
\right\}
= \#
\left\{ \begin{array}{c}
        \text{loc. min. pts of }\\
         R \text{ with } \chi= n
    \end{array}
\right\} = \text{rank}_\bZ\ker( U : \bH^0_{2n} \to \bH^0_{2n-2} ).
\end{equation}
That is, we can recover the number of local minimum points with a given weight from the graded $\bZ[U]$-module structure of $\bH^{0}(C, o)$.\vspace{2mm}

\subsection{Idea of proof of Theorem \ref{thm:MAIN}}\label{subs:idea} We will formulate the idea of the proof in the language of the graded root for the sake of presentation, but the actual proof will only make use of the graded $\bZ[U]$-module structure of the lattice cohomology module.\\

\noindent Let $(C,o)$ be an irreducible plane curve singularity with semigroup $\mathcal{S}_{C,o} = \langle \beta_0,\dots, \beta_g \rangle$ (as in item \textit{(v)} of Proposition \ref{rem:prelim}) and graded root $R(C,o)$. According to item \textit{(i)} of Proposition \ref{rem:prelim}, to prove Theorem \ref{thm:MAIN}, it is sufficient to recover $\mathcal{S}_{C,o}$.

We will define the set of numbers $\mathfrak{N}  \subseteq \bZ_{\geq \min w_0}$ characterized by the property that for any $n \in \mathfrak{N}$ the `truncated root' $R_{\geq n}:= \{\, v\in V \ | \ \chi(v) \geq n\,\}$ is `simple'. By this we mean the following. Define the `trimmed' root $\overline{R}:=R \setminus \{ \text{ leaves of } R\,\}$. We now say that $R_{\geq n}$ is \emph{simple} if and only if $(\overline{R})_{\geq n}$ has exactly one leaf. We are interested in this `simple' part, where the leaves behave in a controlled, separated way, since the local minimum points of the graded root correspond to semigroup elements, which we can then hope to recognize. Indeed, this will turn out to be the case.

Here, a warning is in order. First, note that $[1,\infty) \subseteq \mathfrak{N}$ by item \textit{(4)} of Theorem \ref{thm:prop}. If $\overline{R}$ has more than one leaf then $\mathfrak{N}$ has the form $[e,\infty)$ for some integer $e > \min w_0$. Otherwise, if $\overline{R}$ has only one leaf, then $(\overline{R})_{\geq n}$ has only one leaf for any $n \in \bZ$. However, in this case we set $\mathfrak{N} = [\min w_0, \infty )\cap \mathbb{Z}$ by definition.

As an example consider the semigroup $\mathcal{S}:=\langle 4, 11 \rangle$ with graded root $R$, trimmed root $\overline{R}$ and truncated root $R_{\geq -3}$ depicted below.

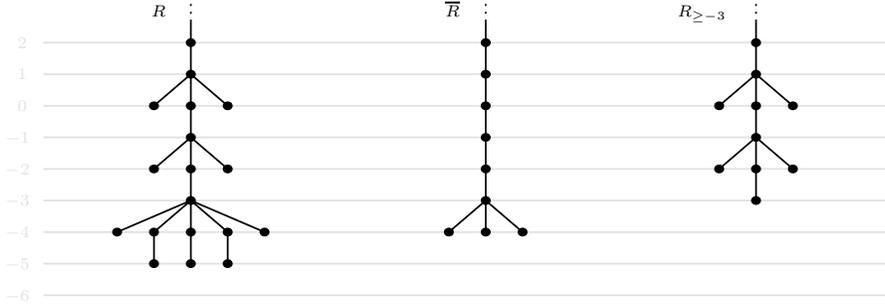
\begin{center}

\tikzset{every picture/.style={line width=0.75pt}} 
\tiny

\tikzset{every picture/.style={line width=0.75pt}} 
\resizebox{12cm}{!}{ 
\begin{tikzpicture}[x=0.75pt,y=0.75pt,yscale=-.6,xscale=.7]

\draw [color={rgb, 255:red, 225; green, 225; blue, 225 }  ,draw opacity=1 ]   (50,60) -- (740,60) ;
\draw [color={rgb, 255:red, 225; green, 225; blue, 225 }  ,draw opacity=1 ]   (50,90) -- (740,90) ;
\draw [color={rgb, 255:red, 225; green, 225; blue, 225 }  ,draw opacity=1 ]   (50,120) -- (740,120) ;
\draw [color={rgb, 255:red, 225; green, 225; blue, 225 }  ,draw opacity=1 ]   (50,150) -- (740,150) ;
\draw [color={rgb, 255:red, 225; green, 225; blue, 225 }  ,draw opacity=1 ]   (50,180) -- (740,180) ;
\draw [color={rgb, 255:red, 225; green, 225; blue, 225 }  ,draw opacity=1 ]   (50,210) -- (740,210) ;
\draw [color={rgb, 255:red, 225; green, 225; blue, 225 }  ,draw opacity=1 ]   (50,240) -- (740,240) ;
\draw [color={rgb, 255:red, 225; green, 225; blue, 225 }  ,draw opacity=1 ]   (50,270) -- (740,270) ;
\draw [color={rgb, 255:red, 225; green, 225; blue, 225 }  ,draw opacity=1 ]   (50,300) -- (740,300) ;
\draw    (170,210) -- (110,240) ;
\draw [shift={(110,240)}, rotate = 153.43] [color={rgb, 255:red, 0; green, 0; blue, 0 }  ][fill={rgb, 255:red, 0; green, 0; blue, 0 }  ][line width=0.75]      (0, 0) circle [x radius= 3.35, y radius= 3.35]   ;
\draw    (170,210) -- (230,240) ;
\draw [shift={(230,240)}, rotate = 26.57] [color={rgb, 255:red, 0; green, 0; blue, 0 }  ][fill={rgb, 255:red, 0; green, 0; blue, 0 }  ][line width=0.75]      (0, 0) circle [x radius= 3.35, y radius= 3.35]   ;
\draw    (170,210) -- (140,240) ;
\draw [shift={(140,240)}, rotate = 135] [color={rgb, 255:red, 0; green, 0; blue, 0 }  ][fill={rgb, 255:red, 0; green, 0; blue, 0 }  ][line width=0.75]      (0, 0) circle [x radius= 3.35, y radius= 3.35]   ;
\draw    (170,210) -- (170,240) ;
\draw [shift={(170,240)}, rotate = 90] [color={rgb, 255:red, 0; green, 0; blue, 0 }  ][fill={rgb, 255:red, 0; green, 0; blue, 0 }  ][line width=0.75]      (0, 0) circle [x radius= 3.35, y radius= 3.35]   ;
\draw    (170,210) -- (200,240) ;
\draw [shift={(200,240)}, rotate = 45] [color={rgb, 255:red, 0; green, 0; blue, 0 }  ][fill={rgb, 255:red, 0; green, 0; blue, 0 }  ][line width=0.75]      (0, 0) circle [x radius= 3.35, y radius= 3.35]   ;
\draw    (200,240) -- (200,270) ;
\draw [shift={(200,270)}, rotate = 90] [color={rgb, 255:red, 0; green, 0; blue, 0 }  ][fill={rgb, 255:red, 0; green, 0; blue, 0 }  ][line width=0.75]      (0, 0) circle [x radius= 3.35, y radius= 3.35]   ;
\draw    (170,240) -- (170,270) ;
\draw [shift={(170,270)}, rotate = 90] [color={rgb, 255:red, 0; green, 0; blue, 0 }  ][fill={rgb, 255:red, 0; green, 0; blue, 0 }  ][line width=0.75]      (0, 0) circle [x radius= 3.35, y radius= 3.35]   ;
\draw    (140,240) -- (140,270) ;
\draw [shift={(140,270)}, rotate = 90] [color={rgb, 255:red, 0; green, 0; blue, 0 }  ][fill={rgb, 255:red, 0; green, 0; blue, 0 }  ][line width=0.75]      (0, 0) circle [x radius= 3.35, y radius= 3.35]   ;
\draw    (170,180) -- (170,210) ;
\draw [shift={(170,210)}, rotate = 90] [color={rgb, 255:red, 0; green, 0; blue, 0 }  ][fill={rgb, 255:red, 0; green, 0; blue, 0 }  ][line width=0.75]      (0, 0) circle [x radius= 3.35, y radius= 3.35]   ;
\draw    (170,150) -- (170,180) ;
\draw [shift={(170,180)}, rotate = 90] [color={rgb, 255:red, 0; green, 0; blue, 0 }  ][fill={rgb, 255:red, 0; green, 0; blue, 0 }  ][line width=0.75]      (0, 0) circle [x radius= 3.35, y radius= 3.35]   ;
\draw    (170,120) -- (170,150) ;
\draw [shift={(170,150)}, rotate = 90] [color={rgb, 255:red, 0; green, 0; blue, 0 }  ][fill={rgb, 255:red, 0; green, 0; blue, 0 }  ][line width=0.75]      (0, 0) circle [x radius= 3.35, y radius= 3.35]   ;
\draw    (170,90) -- (170,120) ;
\draw [shift={(170,120)}, rotate = 90] [color={rgb, 255:red, 0; green, 0; blue, 0 }  ][fill={rgb, 255:red, 0; green, 0; blue, 0 }  ][line width=0.75]      (0, 0) circle [x radius= 3.35, y radius= 3.35]   ;
\draw    (170,60) -- (170,90) ;
\draw [shift={(170,90)}, rotate = 90] [color={rgb, 255:red, 0; green, 0; blue, 0 }  ][fill={rgb, 255:red, 0; green, 0; blue, 0 }  ][line width=0.75]      (0, 0) circle [x radius= 3.35, y radius= 3.35]   ;
\draw    (170,40) -- (170,60) ;
\draw [shift={(170,60)}, rotate = 90] [color={rgb, 255:red, 0; green, 0; blue, 0 }  ][fill={rgb, 255:red, 0; green, 0; blue, 0 }  ][line width=0.75]      (0, 0) circle [x radius= 3.35, y radius= 3.35]   ;
\draw  [dash pattern={on 0.84pt off 2.51pt}]  (170,40) -- (170,20) ;
\draw    (170,150) -- (140,180) ;
\draw [shift={(140,180)}, rotate = 135] [color={rgb, 255:red, 0; green, 0; blue, 0 }  ][fill={rgb, 255:red, 0; green, 0; blue, 0 }  ][line width=0.75]      (0, 0) circle [x radius= 3.35, y radius= 3.35]   ;
\draw    (170,90) -- (140,120) ;
\draw [shift={(140,120)}, rotate = 135] [color={rgb, 255:red, 0; green, 0; blue, 0 }  ][fill={rgb, 255:red, 0; green, 0; blue, 0 }  ][line width=0.75]      (0, 0) circle [x radius= 3.35, y radius= 3.35]   ;
\draw    (170,150) -- (200,180) ;
\draw [shift={(200,180)}, rotate = 45] [color={rgb, 255:red, 0; green, 0; blue, 0 }  ][fill={rgb, 255:red, 0; green, 0; blue, 0 }  ][line width=0.75]      (0, 0) circle [x radius= 3.35, y radius= 3.35]   ;
\draw    (170,90) -- (200,120) ;
\draw [shift={(200,120)}, rotate = 45] [color={rgb, 255:red, 0; green, 0; blue, 0 }  ][fill={rgb, 255:red, 0; green, 0; blue, 0 }  ][line width=0.75]      (0, 0) circle [x radius= 3.35, y radius= 3.35]   ;
\draw    (410,210) -- (380,240) ;
\draw [shift={(380,240)}, rotate = 135] [color={rgb, 255:red, 0; green, 0; blue, 0 }  ][fill={rgb, 255:red, 0; green, 0; blue, 0 }  ][line width=0.75]      (0, 0) circle [x radius= 3.35, y radius= 3.35]   ;
\draw    (410,210) -- (410,240) ;
\draw [shift={(410,240)}, rotate = 90] [color={rgb, 255:red, 0; green, 0; blue, 0 }  ][fill={rgb, 255:red, 0; green, 0; blue, 0 }  ][line width=0.75]      (0, 0) circle [x radius= 3.35, y radius= 3.35]   ;
\draw    (410,210) -- (440,240) ;
\draw [shift={(440,240)}, rotate = 45] [color={rgb, 255:red, 0; green, 0; blue, 0 }  ][fill={rgb, 255:red, 0; green, 0; blue, 0 }  ][line width=0.75]      (0, 0) circle [x radius= 3.35, y radius= 3.35]   ;
\draw    (410,180) -- (410,210) ;
\draw [shift={(410,210)}, rotate = 90] [color={rgb, 255:red, 0; green, 0; blue, 0 }  ][fill={rgb, 255:red, 0; green, 0; blue, 0 }  ][line width=0.75]      (0, 0) circle [x radius= 3.35, y radius= 3.35]   ;
\draw    (410,150) -- (410,180) ;
\draw [shift={(410,180)}, rotate = 90] [color={rgb, 255:red, 0; green, 0; blue, 0 }  ][fill={rgb, 255:red, 0; green, 0; blue, 0 }  ][line width=0.75]      (0, 0) circle [x radius= 3.35, y radius= 3.35]   ;
\draw    (410,120) -- (410,150) ;
\draw [shift={(410,150)}, rotate = 90] [color={rgb, 255:red, 0; green, 0; blue, 0 }  ][fill={rgb, 255:red, 0; green, 0; blue, 0 }  ][line width=0.75]      (0, 0) circle [x radius= 3.35, y radius= 3.35]   ;
\draw    (410,90) -- (410,120) ;
\draw [shift={(410,120)}, rotate = 90] [color={rgb, 255:red, 0; green, 0; blue, 0 }  ][fill={rgb, 255:red, 0; green, 0; blue, 0 }  ][line width=0.75]      (0, 0) circle [x radius= 3.35, y radius= 3.35]   ;
\draw    (410,60) -- (410,90) ;
\draw [shift={(410,90)}, rotate = 90] [color={rgb, 255:red, 0; green, 0; blue, 0 }  ][fill={rgb, 255:red, 0; green, 0; blue, 0 }  ][line width=0.75]      (0, 0) circle [x radius= 3.35, y radius= 3.35]   ;
\draw    (410,40) -- (410,60) ;
\draw [shift={(410,60)}, rotate = 90] [color={rgb, 255:red, 0; green, 0; blue, 0 }  ][fill={rgb, 255:red, 0; green, 0; blue, 0 }  ][line width=0.75]      (0, 0) circle [x radius= 3.35, y radius= 3.35]   ;
\draw  [dash pattern={on 0.84pt off 2.51pt}]  (410,40) -- (410,20) ;
\draw    (630,180) -- (630,210) ;
\draw [shift={(630,210)}, rotate = 90] [color={rgb, 255:red, 0; green, 0; blue, 0 }  ][fill={rgb, 255:red, 0; green, 0; blue, 0 }  ][line width=0.75]      (0, 0) circle [x radius= 3.35, y radius= 3.35]   ;
\draw    (630,150) -- (630,180) ;
\draw [shift={(630,180)}, rotate = 90] [color={rgb, 255:red, 0; green, 0; blue, 0 }  ][fill={rgb, 255:red, 0; green, 0; blue, 0 }  ][line width=0.75]      (0, 0) circle [x radius= 3.35, y radius= 3.35]   ;
\draw    (630,120) -- (630,150) ;
\draw [shift={(630,150)}, rotate = 90] [color={rgb, 255:red, 0; green, 0; blue, 0 }  ][fill={rgb, 255:red, 0; green, 0; blue, 0 }  ][line width=0.75]      (0, 0) circle [x radius= 3.35, y radius= 3.35]   ;
\draw    (630,90) -- (630,120) ;
\draw [shift={(630,120)}, rotate = 90] [color={rgb, 255:red, 0; green, 0; blue, 0 }  ][fill={rgb, 255:red, 0; green, 0; blue, 0 }  ][line width=0.75]      (0, 0) circle [x radius= 3.35, y radius= 3.35]   ;
\draw    (630,60) -- (630,90) ;
\draw [shift={(630,90)}, rotate = 90] [color={rgb, 255:red, 0; green, 0; blue, 0 }  ][fill={rgb, 255:red, 0; green, 0; blue, 0 }  ][line width=0.75]      (0, 0) circle [x radius= 3.35, y radius= 3.35]   ;
\draw    (630,40) -- (630,60) ;
\draw [shift={(630,60)}, rotate = 90] [color={rgb, 255:red, 0; green, 0; blue, 0 }  ][fill={rgb, 255:red, 0; green, 0; blue, 0 }  ][line width=0.75]      (0, 0) circle [x radius= 3.35, y radius= 3.35]   ;
\draw  [dash pattern={on 0.84pt off 2.51pt}]  (630,40) -- (630,20) ;
\draw    (630,150) -- (600,180) ;
\draw [shift={(600,180)}, rotate = 135] [color={rgb, 255:red, 0; green, 0; blue, 0 }  ][fill={rgb, 255:red, 0; green, 0; blue, 0 }  ][line width=0.75]      (0, 0) circle [x radius= 3.35, y radius= 3.35]   ;
\draw    (630,90) -- (600,120) ;
\draw [shift={(600,120)}, rotate = 135] [color={rgb, 255:red, 0; green, 0; blue, 0 }  ][fill={rgb, 255:red, 0; green, 0; blue, 0 }  ][line width=0.75]      (0, 0) circle [x radius= 3.35, y radius= 3.35]   ;
\draw    (630,150) -- (660,180) ;
\draw [shift={(660,180)}, rotate = 45] [color={rgb, 255:red, 0; green, 0; blue, 0 }  ][fill={rgb, 255:red, 0; green, 0; blue, 0 }  ][line width=0.75]      (0, 0) circle [x radius= 3.35, y radius= 3.35]   ;
\draw    (630,90) -- (660,120) ;
\draw [shift={(660,120)}, rotate = 45] [color={rgb, 255:red, 0; green, 0; blue, 0 }  ][fill={rgb, 255:red, 0; green, 0; blue, 0 }  ][line width=0.75]      (0, 0) circle [x radius= 3.35, y radius= 3.35]   ;

\draw (27,52.4) node [anchor=north west][inner sep=0.75pt]  [color={rgb, 255:red, 225; green, 225; blue, 225 }  ,opacity=1 ]  {$2$};
\draw (27,82.4) node [anchor=north west][inner sep=0.75pt]  [color={rgb, 255:red, 225; green, 225; blue, 225 }  ,opacity=1 ]  {$1$};
\draw (27,112.4) node [anchor=north west][inner sep=0.75pt]  [color={rgb, 255:red, 225; green, 225; blue, 225 }  ,opacity=1 ]  {$0$};
\draw (17,142.4) node [anchor=north west][inner sep=0.75pt]  [color={rgb, 255:red, 225; green, 225; blue, 225 }  ,opacity=1 ]  {$-1$};
\draw (17,172.4) node [anchor=north west][inner sep=0.75pt]  [color={rgb, 255:red, 225; green, 225; blue, 225 }  ,opacity=1 ]  {$-2$};
\draw (17,202.4) node [anchor=north west][inner sep=0.75pt]  [color={rgb, 255:red, 225; green, 225; blue, 225 }  ,opacity=1 ]  {$-3$};
\draw (17,232.4) node [anchor=north west][inner sep=0.75pt]  [color={rgb, 255:red, 225; green, 225; blue, 225 }  ,opacity=1 ]  {$-4$};
\draw (17,262.4) node [anchor=north west][inner sep=0.75pt]  [color={rgb, 255:red, 225; green, 225; blue, 225 }  ,opacity=1 ]  {$-5$};
\draw (17,292.4) node [anchor=north west][inner sep=0.75pt]  [color={rgb, 255:red, 225; green, 225; blue, 225 }  ,opacity=1 ]  {$-6$};
\draw (136,22.4) node [anchor=north west][inner sep=0.75pt]    {$R$};
\draw (375,18.4) node [anchor=north west][inner sep=0.75pt]    {$\overline{R}$};
\draw (564,22.4) node [anchor=north west][inner sep=0.75pt]    {$R_{\geq -3}$};

\end{tikzpicture}
}

\end{center}

\noindent It is easy to see that $\mathfrak{N} = [-3,\infty)\cap \mathbb{Z}$, as $(\overline{R})_{\geq -3}$ is an infinite path having exactly one vertex of degree $n$ for all $n \in [-3,\infty)\cap \mathbb{Z}$, whereas  $(\overline{R})_{\geq -4} = \overline{R}$ has three leaves (of degree $\chi=-4$).

As already hinted before, the significance of this simplicity lies in the fact that for any $n \in \mathfrak{N}$, the root $R_{\geq n}$ uniquely determines the elements $s \in \mathcal{S}_{C,o}$ with $w(s) \geq n$. The \emph{initial part} of the semigroup is then defined as $\mathcal{E} := \{\, s \in \mathcal{S}_{C,o}\ | \ s \leq \delta \text{ and } w(s) \geq \min \mathfrak{N} \,\}$. The name is justified by the fact that it has the form $\mathcal{E}= \mathcal{S}_{C,o}\cap [0,d]$ for some $d \in \bZ_{\geq 0}$.

As a continuation of the above example, consider the subset $\mathcal{S}' \subset \bZ_{\geq 0}$ given as
$$\mathcal{S}' = \{\,0,\,4,\,9,\,10,\,12,\,15,\,16,\,18,\,21,\,22,\,23,\,24,\,26,\,27,\,28,\,k \geq 30 \,\}. $$
Note that the set $\mathcal{S}'$ is not a semigroup. Nonetheless one may define its weight function $w'$ (and hence its graded root $R'$) according to item \textit{(vi)} of Proposition \ref{rem:prelim}. The image below depicts the semigroup $\mathcal{S} = \langle 4,11\rangle$ in blue and the set $\mathcal{S}'$ in red. Shown underneath are their respective weight functions, and their common graded root $R \cong R'$ is represented on the right.
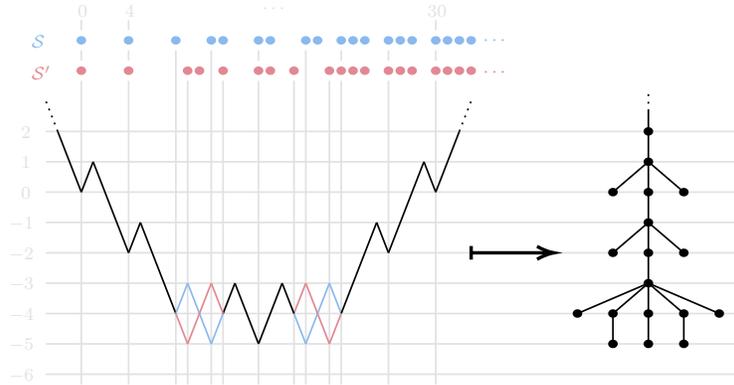
\vspace{2mm}

\begin{center}
\tikzset{every picture/.style={line width=0.75pt}} 

\resizebox{10cm}{!}{ 
\scriptsize
\vspace{2mm}
}
\end{center}

\noindent This shows that the graded root is insufficient to distinguish between arbitrary cofinite subsets of the natural numbers. On the other hand we have $\mathfrak{N}=[-3,\infty)\cap \mathbb{Z}$ as mentioned earlier and indeed, the initial parts $\mathcal{E} = \mathcal{E}' = \{0,4\}$ are uniquely determined by $R \cong R'$.

It will turn out that, in the case of graded roots associated to semigroups of irreducible plane curve singularities, the initial part $\mathcal{E}$ contains all generators of $\mathcal{S}_{C,o}$ except for maybe $\beta_g$. This is sufficient to reconstruct the entire numerical semigroup $\mathcal{S}_{C,o}$, since it is uniquely determined by $\beta_0,\dots,\beta_{g-1}$ and $\delta(C,o)$ (which is also computable by item \textit{(viii)} of Proposition \ref{rem:prelim}).

In the previous example for instance we already know $4 \in \mathcal{S}$ and if we proved, that we have a single generator left, we could compute it from the delta invariant.

\subsection{The initial part of the semigroup} In this subsection we will prove that we can read off the initial part $\mathcal{E}$ of the semigroup $\mathcal{S}_{C, o}$ corresponding to the simple part of the graded root containing only simple leaves.

\begin{definition}\label{def:initialpart}
    Let $\bM_{\geq 2n}$ be the $\bZ[U]$-submodule of $\bH^0$ generated by $\bH^0_{\geq 2n} := \bigoplus_{k \geq n}\bH^0_{2k}$. Note that $\bH^0_{\geq 2n}$ is not a $\bZ[U]$-submodule of $\bH^0$ in general, since $U : \bH^0_{2n} \to \bH^0_{2n-2}$ might be nonzero. In fact, in the direct sum decomposition of Proposition \ref{prop:directsumdecomp}, $\mathbb{M}_{\geq 2n}$ corresponds to the submodule $T_{2 \min w_0}^{\infty} \oplus \bigoplus_{n_k \geq n}T_{m_k}^{n_k} \leq \mathbb{H}^0(C, o)$.
    
    Let us consider the set $\mathfrak{N} \subset \mathbb{Z}_{\geq {\rm min} w}$ of numbers $n$ such that the $U$-action is trivial on the quotient module $\bM_{\geq 2n} \big/ T_{2\min w_0}^\infty$. For fixed $n \geq \min w_0$, this condition is equivalent to any of the following ones:
    \begin{itemize}
        \item the quotient module $\bM_{\geq 2n} \big/ T_{2\min w_0}^\infty$ cannot contain summands of form $T_{2m_k}^{2n_k}$ with $n_k \neq m_k$;
        \item the trimmed and then truncated root $(\overline{R})_{\geq n}$ has exactly one leaf (c.f. subsection \ref{subs:idea}); 
        \item $U \cdot \bM_{\geq 2n} \subset  T_{2\min w_0}^\infty$ (this clearly shows that if $n \in \mathfrak{N}$, then $n+1 \in \mathfrak{N}$ as well).
    \end{itemize} 
    Let us denote the minimal element of $\mathfrak{N}$ by $e$. If we consider the unique (up to isomorphism) direct sum decomposition $\bH^0(C,o) \cong T_{2\min w_0}^\infty\oplus \bigoplus_k T_{2m_k}^{2n_k}$ of Proposition \ref{prop:directsumdecomp}, then 
    \begin{equation}\label{eq:eadirsumdecompbol}
        e=\min \{\, n \in \mathbb{Z}_{\geq \min w_0}\ | \ \forall n_k \geq n:\  m_k =n_k\}
    \end{equation}
    For any integer $n \in \mathfrak{N}$, we define the \emph{$n^{th}$ initial part} of the semigroup $\mathcal{S}_{C, o}$ as    
    $$\mathcal{E}_n : = \{\, s \in \mathcal{S}_{C, o}\ | \ w(s) \geq n \text{ and } s \leq \delta \,\} \text{ and set } \mathcal{E}:= \mathcal{E}_e. $$
    Altough not evident, we will show later in Proporsition \ref{prop:kezdoszelet}, that for every $n \in \mathfrak{N}$, the set $\mathcal{E}_n$ is indeed of the form $\mathcal{S}_{C, o} \cap [0, d_n]$ for some $d_n \in \mathbb{N}$.
\end{definition}

\noindent It clearly follows from the Nonpositivity Theorem (Theorem \ref{thm:prop} \textit{(4)}) that $[1,\infty) \cap \mathbb{Z} \subseteq \mathfrak{N}$. Moreover, we have the following statement.

\begin{prop}\label{prop:0inN}
     $0 \in \mathfrak{N}$ and, thus, $\mathcal{E}_0$ is defined and has $0 \in \mathcal{E}_0$.
\end{prop}

\noindent Before proving Proposition \ref{prop:0inN} we show the following useful technical lemma.

\begin{lemma}\label{lem:ablakos}
    Let $\mathcal{S}$ be a symmetric numerical semigroup with multiplicity $m$ and let $\ell \leq \delta$ be an integer. Then
    $$\max \big( w\big|_{[\ell,\min(\ell+m-1,\delta)]} \big) =  \max\big(w\big|_{[\ell,\delta]}\big).$$
\end{lemma}

\begin{proof}
    Assume that $\ell+m-1 < \delta$, otherwise there is nothing to prove. Given any lattice point $\ell' \in \big[\ell+m, \min(\ell+2m-1,\delta)\big]$, we claim that $w(\ell'-m) \geq w(\ell')$. In order to prove it, we assume by contradiction that $w(\ell'-m) < w(\ell')$. By part \textit{(vi)} of Proposition \ref{rem:prelim}, this is equivalent to
    $\#\big( \mathcal{S}\cap[\ell'-m,\ell'-1] \big) > m/2$. On the other hand, we have that
    \begin{equation}\label{eq:this}
        \#\big( \mathcal{S}\cap[\ell'-m,\ell'-1]\big) \leq \#\big( \mathcal{S}\cap[\ell'-m+1,\ell']\big) \leq \dots \leq \#\big( \mathcal{S}\cap[c-\ell',c-1-\ell'+m]\big).
    \end{equation}  
    Indeed, since all of these intervals have length $m \in \mathcal{S}$, if $\ell' -m +k \in \mathcal{S}$, then so is $\ell' +k \in \mathcal{S}$ for any $k \geq 0$. However, equation (\ref{eq:this}) contradicts the symmetry (\ref{eq:Gorsym}) of $\mathcal{S}$ as
    \begin{align*}
        \frac{m}{2}< \#\big( \mathcal{S}\cap[\ell'-m,\ell'-1] \big) \leq &\, \#\big( \mathcal{S}\cap[c-\ell',c-1-\ell'+m]\big)\\
        = & \, m-\#\big( \mathcal{S}\cap[\ell'-m,\ell'-1] \big) < m-\frac{m}{2} = \frac{m}{2}.
    \end{align*}
    It follows that $\max\big( w\big|_{[\ell,\ell+m-1]}\big) =\max\big( w\big|_{[\ell,\min(\ell + 2m-1,\delta)]}\big)$. Continuing this process inductively we reach the desired equality $\max\big(w\big|_{[\ell,\ell+m-1]}\big) = \max\big( w\big|_{[\ell,\delta]}\big)$.
\end{proof}

\begin{proof}[Proof of Proposition \ref{prop:0inN}]
    If the multiplicity $m=2$, then by Proposition 2.7.3 (c) in \cite{StrucProp} we have $\mathbb{H}^0_{2n}=0$ for all $n <0$ (more precisely: $w \geq 0$) and hence $U \cdot \mathbb{H}^0_0=0$. Therefore $U \cdot \mathbb{M}_{\geq 0}= U \cdot \mathbb{M}_{\geq 1} \subset T_{0}^\infty$, so $0 \in \mathfrak{N}$. 

    If $m \geq 3$, then $w_0(\ell) = 2-\ell$ for every $\ell \in [2, m] \cap \mathbb{Z}$, whereas $w_0(m+1) = 3-m \leq0$. Therefore, $\max \big(w \big|_{[2, m+1]}\big)=0$, and hence, by Lemma \ref{lem:ablakos}, $\max \big( w\big|_{[2, \delta]}\big)=0$. This, and the Gorenstein symmetry (\ref{eq:Gorsym}), implies that in this case $S_0 = \{ 0 \} \cup [2, c-2] \cup \{c\}$, so clearly $U \cdot \mathbb{M}_{\geq 0} \subset T_{2 \min w_0}^{\infty}$.
\end{proof}

\noindent Next, we will characterize the $S_n$-spaces with $n \in \mathfrak{N}$. First we need the following simple lemma.

\begin{lemma} \label{lem:interval}\,

\begin{itemize}
    \item[a)] If the interval $[s, s+k]\ (s, k \in \mathbb{Z}_{\geq 0})$ is a connected component of $ S_n$, with $k>0$, then $[s, s+k]\cap S_{n-1} \neq \emptyset$ (hence the corresponding generator of $\mathbb{H}^0_{2n}$ is not in the kernel of the $U$-action).
    \item[b)] Similarly, if $(s, s+k) \ (s, k \in \mathbb{Z}_{\geq 0})$ is a connected component of $\mathbb{R} \setminus S_n$, with $k >2$, then $(s, s+k) \cap \big( \mathbb{R} \setminus S_{n+1}\big) \neq \emptyset$. 
\end{itemize}
    
\end{lemma}

\begin{proof}
    We prove statement \textit{a)}.
    Since $s+k+1 \not\in S_n$, it follows from formula (\ref{eq:w_0valtozasa}), that $w_0(s+k+1)=n+1$, $w_0(s+k)=n$. Furthermore, $s+k-1 \in S_n$ and $w_0(s+k-1)=n-1$. Hence $s+k-1 \in [s, s+k]\cap S_{n-1}$. Statement \textit{b)} follows analogously.
\end{proof}

\begin{lemma} \label{lem:Snalakja}
    If $n \in \mathfrak{N}$, then $S_n$ has the form 
    $$S_n = \{s_n, s_n +2, \ldots,  s_n+2k_n-2\}\sqcup [s_n+2k_n,c-s_n-2{k_n}] \sqcup \{c-s_n-2{k_n}+2,\ldots,c-s_{n}\},$$
where $k_n = \lfloor {\rm rank}_\bZ\bH^0_{2n}/2\rfloor$ and $s_n \leq \delta$. Note, that the interval $[s_n+2k_n,c-s_n-2{k_n}]$ might  be just a single point or even empty. In the latter case $s_n+2k_n = c-s_n-2k_n+2$.
\end{lemma}

\begin{proof}
 Clearly, the space $S_n$ is a disjoint union of closed intervals. If $S_n$ contained more than one nontrivial interval, Lemma \ref{lem:interval} would imply that these intersected $S_{n-1}$ in multiple components.
    Thus, the triviality of the $U$-action on $\bM_{\geq 2n} \big/ T_{2\min w_0}^\infty$ implies that $S_n$ is the disjoint union of some points and at most a single interval, i.e. by Gorenstein symmetry (cf. (\ref{eq:Gorsym}))
    $$S_n = \{s^n_0,\dots,s^n_{k_n-1}\}\sqcup [s_{k_n}^n,c-s_{k_n}^n] \sqcup \{c-s^n_{k_n-1},\dots,c-s^n_{0}\},$$
where $k_n = \lfloor \text{rank}_\bZ\bH^0_{2n}/2 \rfloor$. Notice that for any $s \in S_n \cap \mathbb{Z}_{>0}$, by formula (\ref{eq:w_0valtozasa}) we have $[s-1, s+1] \subset S_{n+1}$.  Hence, item \textit{b)} of Lemma \ref{lem:interval} and the triviality of the $U$-action on $\bM_{\geq 2n+2} \big/ T_{2\min w_0}^\infty$ shows that
    $$|s_{i-1}^n-s_i^n| = 2 \text{ for all } i = 1, \dots, k_{n}.$$
Indeed, otherwise the space $S_{n+1}$ (with $n+1 \in \mathfrak{N}$) would contain more than one non-trivial interval. In summary,
    $$S_n \cap [0,\delta] = \{s^n_{0}, \,s^n_{0}+2,\, \dots,\, s^n_{0} + 2(k_n-1) \} \sqcup [s^n_{0}+2k_n, \delta],$$
    so we take $s_0^n$ as the desired $s_n$.
\end{proof}

Note that the number $\text{rank}_{\bZ}\bH^0_{2n}(C,o)$ can indeed be both even and odd, see Example \ref{ex:6,15,31} and the graded root shown before subsection \ref{subs:elate} for instance.

We can now prove the following characterization:

\begin{prop}\label{prop:ninNhaSn}
    $n \in \mathfrak{N}$ if and only if for all $n' \geq n$, the space $S_{n'}$ has the form  form $$S_{n'} = \{s_{n'}, s_{n'} +2, \ldots,  s_{n'}+2k_{n'}-2\}\sqcup [s_{n'}+2k_{n'},c-s_{n'}-2{k_{n'}}] \sqcup \{c-s_{n'}-2{k_{n'}}+2,\ldots,c-s_{n'}\},$$
where $k_{n'} = \lfloor {\rm rank}_\bZ\bH^0_{2n'}/2 \rfloor$ and $s_{n'} \leq \delta$. Similarly as above, the interval $[s_{n'}+2k_{n'},c-s_{n'}-2{k_{n'}}]$ may be a single point or empty. In the latter case $s_{n'}+2k_{n'} = c-s_{n'}-2k_{n'}+2$.
\end{prop}

\begin{proof}
    The necessity of the condition is clear from Lemma \ref{lem:Snalakja}. The sufficiency comes from the fact that, through the correspondence (\ref{eq:locmin}) between local minimum points of the weight function and kernel of the $U$-action, $U \cdot \mathbb{H}^0_{2n'} \leq T_{2\min w_0}^{\infty}$ for all $n' \geq n$, hence $U \cdot \mathbb{M}_{\geq 2n} \leq T_{2\min w_0}^\infty$. 
\end{proof}

\begin{example} \label{ex:6,15,31} Consider the irreducible plane curve singularity $(C,o) \subset (\bC^2,o)$ given by the Puiseux parametrization $t \mapsto (t^6,t^{15}+t^{16})$. Its semigroup of values $\mathcal{S}_{C,o}$ is minimally generated by the set $\{6,15,31\}$ and has conductor $c=72$.

The diagram below depicts the semigroup $\mathcal{S}_{C,o}$, the space $S_n$ for $-15 \leq n \leq 2$ and the graded root $R(C,o)$. Clearly, $\mathfrak{N} = [-12,\infty)\cap \mathbb{Z}$ and indeed, $S_n = \{s_n\} \sqcup [s_n+2,c-s_n-2] \sqcup \{c-s_n\}$ for all $n\in \mathfrak{N}$ as claimed in Proposition \ref{prop:ninNhaSn}. \vspace{3mm}
    \begin{center}
        
\tiny

\tikzset{every picture/.style={line width=0.75pt}} 
\tiny

\tikzset{every picture/.style={line width=0.75pt}} 
\resizebox{12cm}{!}{ 

}
\vspace{1mm}

    \end{center}
\end{example}

\begin{corollary} \label{cor:nweightedS}
    Suppose that $n \in \mathfrak{N}$. According to Lemma \ref{lem:Snalakja}, the space $S_n$ has the form $$S_n = \{s_n, s_n +2, \ldots,  s_n+2k_n-2\}\sqcup [s_n+2k_n,c-s_n-2{k_n}] \sqcup \{c-s_n-2{k_n}+2,\ldots,c-s_{n}\},$$
    where $k_{n} = \lfloor {\rm rank}_\bZ\bH^0_{2n}/2 \rfloor$. Then we have the following statements.
    \begin{itemize}
        \item[a)] If $n>\min w_0$, then the elements of $\mathcal{S}_{C, o} \cap [0, \delta]$ of weight $n \in \mathfrak{N}$ are precisely the numbers $s_n + 2j$ with $0 \leq j <k_n$. (If $k_n=0$, then there are no semigroup elements of weight $n$.)
        \item[b)] If $n=e=\min w_0$, then the elements of $\mathcal{S}_{C, o} \cap [0, \delta]$ of weight $e \in \mathfrak{N}$ are precisely the numbers $s_e + 2j$,  where $0 \leq j < k_n$ if $ {\rm rank}_\bZ\bH^0_{2e}$ even, and $0 \leq j \leq k_n$ if $ {\rm rank}_\bZ\bH^0_{2e}$ odd 
    \end{itemize}
\end{corollary}

\begin{proof}
    a) By (\ref{eq:w_0valtozasa}), $w_0(s_n + 2j)=n$ and $w_0(s_n + 2j+1)=n+1$ for all $0 \leq j < k_n$, hence $s_n + 2j \in \mathcal{S} \cap [0, \delta]$. On the other hand, for any $s \in \mathcal{S} \cap [0, \delta]$ we must have $s \in S_n, \ s+1 \notin S_n$, hence, they must be of the prescribed form. Indeed, $s_n + 2k_n=c-s_n-2k_n=\delta$ can only happen if $n = e=\min w_0$, since in this case $S_n$ does not contain any interval and, thus, $S_{n-1} = \emptyset$.

    b) If $n=e=\min w_0$, then by Lemma \ref{lem:interval} \textit{a)} $S_e$ consists of discrete points of weight $e$, thus, each a semigroup element. The ones in the interval $[0, \delta]$ are described above.
\end{proof}

\begin{remark} \label{rem:delta}
    Notice that the case of $ {\rm rank}_\bZ\bH^0_{2\min w_0}$ odd is equivalent to $w(\delta) = \min w_0$ (by Gorenstein symmetry (\ref{eq:Gorsym})), and, if $\min w_0=e \in \mathfrak{N}$, it is also characterized by $s_e+2k_e=\delta$.
\end{remark}

\begin{prop} \label{prop:kezdoszelet}
    For any $n \in \mathfrak{N}$, $\mathcal{E}_n$ is of the form $\mathcal{S}_{C, o} \cap [0, s_n +2k_n]$.
\end{prop}

\begin{proof}
    On one hand, since $s_{n-1} + 2k_{n-1} < s_n + 2k_n$ for all $n \in \mathfrak{N}$, Corollary \ref{cor:nweightedS} implies that $\mathcal{E}_n \subset \mathcal{S}_{C, o} \cap [0, s_n +2k_n]$. On the other hand, if $s \in \mathcal{S}_{C, o}$ has weight $w_0(s) < n$, then $s \in S_n$, even more, by Lemma \ref{lem:Snalakja}, $s \in [s_n + 2k_n, c-s_n-2k_n]$, hence $s > s_n +2k_n$, because $w_0(s_n+2k_n)=n$. 
\end{proof}

\begin{prop}\label{prop:emeghat}
    For any $n \in \mathfrak{N}$, with $n \leq 0$, the set $\mathcal{E}_n$ and the graded Abelian group $\bH^0_{\geq 2n}$ uniquely determine each other if $\min w$ is known. In particular, $\mathcal{E}_n$ is determined by the graded $\bZ[U]$-module $\bH^0$.
\end{prop}

\begin{proof} First we show that $\bH^0_{\geq 2n}$ determines $\mathcal{E}_n$. From Corollaries \ref{cor:nweightedS} and \ref{cor:w(s)leq0} we get that 
\begin{equation*} \label{eq:calE_n}
    \mathcal{E}_n = \begin{cases}
        \{\, s_{n'}+2j \ | \ n \leq n' \leq 0, 0 \leq j < \ k_{n'} \, \} \cup\{s_n+2k_n\} & \text{ if }  n=\min w_0 \text{ and }  {\rm rank}_\bZ\bH^0_{2n} \text{ odd}; \\
        \{\, s_{n'}+2j \ | \ n \leq n' \leq 0, 0 \leq j < \ k_{n'} \, \} & \text{ otherwise}.
    \end{cases} 
\end{equation*}
(Equivalently, by Remark \ref{rem:delta}, $\mathcal{E}_n = \{\, s_{n'}+2j \ | \ n \leq n' \leq 0, 0 \leq j < \ k_{n'} \, \} \cup\big(\{s_n+2k_n\}\cap \{\delta\}\big)$.) 
Also, the numbers $s_n$ are determined recursively: $s_0 = 0$ and $s_{n-1} = s_n+2k_n+1=s_n + {\rm rank}_\mathbb{Z}\mathbb{H}^0_{2n}$ (Indeed, if $S_n$ contains a nontrivial interval $[s_n +2k_n, c-s_n-2k_n]$, then $s_n +2k_n +1 \in S_{n-1}$ and it is the smallest element. Otherwise, if $S_n$ is just a union of discrete points, then, by formula (\ref{eq:w_0valtozasa}) $S_{n-1} = \emptyset$ and thus $n=e=\min w_0$ with $n-1 \notin \mathfrak{N}$). We remark, that in the case of  $n=\min w_0$ with ${\rm rank}_\bZ\bH^0_{2n}$ odd, we get back item \textit{(viii)} of Proposition \ref{rem:prelim}.

Secondly, $\mathcal{E}_n$ determines $\mathbb{H}_{\geq 2n}$ as follows: since $\mathcal{E}_n = \mathcal{S}_{C, o} \cap [0,s_n+2k_n]$, we can define the modified weight function $\tilde{w}$ as
$$\tilde{w}(x):=\begin{cases}
    \ w(x) &\text{ if } x \in [0,s_n+2k_n] \sqcup [c-s_n-2k_n,\infty) \\
    \ \hspace{2.5mm}n &\text{ if } x \in [s_n+2k_n,c-s_n-2k_n]
\end{cases}.$$
The lattice cohomology of $\tilde{w}$ is isomorphic to $\bH^0 \big/ \bH^0_{< 2n}$. Forgetting the $U$-action yields $\bH^0_{\geq 2n}$.
\end{proof}

\begin{example}\cite[Example 1.6]{LCandmult}
    The local minimum points of the weight function $w_0$ are always elements of the semigroup $\mathcal{S}_{C, o}$. The essence of Proposition \ref{prop:emeghat} is the fact, that in $\mathcal{E}$ the natural ordering of the semigroup elements reflects the ordering of the local minimum values associated to them in the sense that $s_1 < s_2$ implies $w(s_1) \geq w(s_2)$. If this was true for \emph{any} two semigroup elements less than $\delta$, then reconstructing $\mathcal{S}_{C, o}$ from $R(C, o)$ would be straightforward.

    However, the image below provides an example showing that this `convexity condition' of local minimum points does not always hold. It depicts the semigroup $\mathcal{S}=\langle 11,14 \rangle$ with conductor $c=130$ and its associated weight function. One can see that $55 < 58 \in \mathcal{S}\cap[0,\delta)$ whereas $w(55) = -27 < w(58) = -26$.
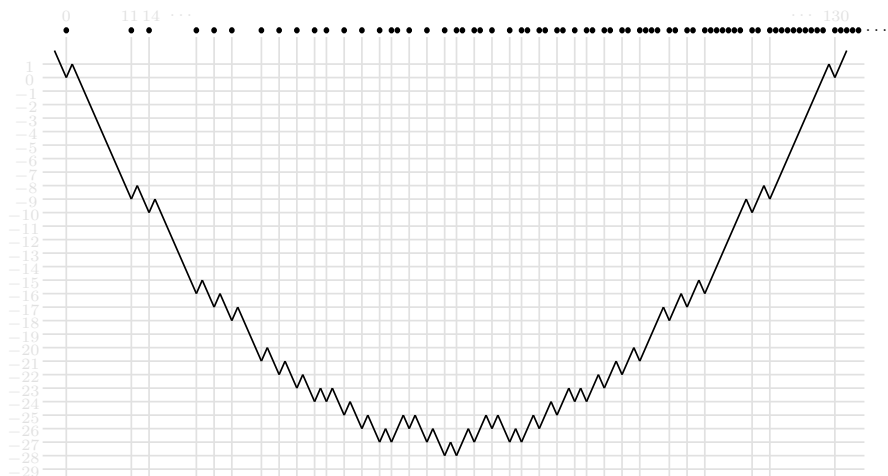
\vspace{2mm}
    \begin{center}

\tikzset{every picture/.style={line width=0.75pt}} 
\resizebox{12cm}{!}{ 
\tiny

}

    \end{center}
\end{example}

\noindent We can easily describe the algorithm how to read off $\mathcal{E}_n$, for any $n \in \mathfrak{N}$, from the lattice cohomology $\mathbb{Z}[U]$-module $\mathbb{H}^0(C, o)$ of a non-smooth irreducible plane curve singularity $(C, o) \subset (\mathbb{C}^2, o)$ or from the graded root $R(C, o)$:

\begin{enumerate}
    \item Start at level $n=0$ with $s_0=0$. Let $k_0:=\lfloor{\rm rank}_{\mathbb{Z}}\big(\mathbb{H}^0_0\big)/2\rfloor \ (\geq 1$ by the contraposition of \cite[Example 4.6.1]{AgNeCurve} and Gorenstein symmetry (\ref{eq:Gorsym})). If $0=e=\min w_0$, move to step (4), otherwise $\mathcal{E}_0=\{ s_0, s_0+2, \ldots, s_0+2k_0-2\}=\mathcal{S}_{C, o} \cap [0, s_0+2k_0]$.
    \item If $0 \neq e$, consider $\mathbb{H}^0_{-2}$. Set $s_{-1}:=s_{0}+2k_0+1=s_0+{\rm rank}_\mathbb{Z}\mathbb{H}^0_{0}$ and $k_{-1}:=\lfloor {\rm rank}_{\mathbb{Z}}\big( \mathbb{H}^0_{-2}\big)\rfloor$. If $-1=e=\min w_0$, move to step (4), otherwise $\mathcal{E}_{-1}=\mathcal{E}_0 \cup \{ s_{-1}, s_{-1}+2, \ldots, s_{-1}+2k_{-1}-2\}=\mathcal{S}_{C, o} \cap [0, s_{-1}+2k_{-1}]$. Notice, that $\{ s_{-1}, s_{-1}+2, \ldots, s_{-1}+2k_{-1}-2\}$ might be empty.
    \item If $-1 \neq e$, repeat step (2) with the indices decreased by $1$: swap $0 \mapsto -1$ and $-1 \mapsto -2$ to get $\mathcal{E}_{-2}$. Decrease the index, repeat step (2), descend, repeat etc., until arrived to $e$, where, if $e \neq \min w_0$, get $\mathcal{E}= \mathcal{E}_e = \mathcal{S}_{C,o} \cap [0, s_e+2k_e]$.
    \item For $e=\min w_0$ set $\mathcal{E}= \mathcal{E}_e = \mathcal{E}_{e+1} \cup \{s_e, s_e+2, \ldots, s_e+2\lceil{\rm rank}_{\mathbb{Z}}H^0_{2e}/2\rceil-2\} =\mathcal{S}_{C,o} \cap [0, \delta]$.
\end{enumerate}

\noindent In practice, one usually uses a slightly different, but equivalent, algorithm which we will now describe in terms of the graded root $R(C,o)$.

Let $n \in \mathfrak{N}$ and consider the truncated root $R_{\geq n}$. Delete all of its edges and for each $n' \geq n$, and delete half of the vertices (rounded up, except for the level $\chi = \min w_0$, where rounded down) of weight $n'$. Obtain thus a graded set $V_n$. Enumerate its vertices as $V_n=\{\,v_0,\, \dots,\,v_k\,\}$ in such a way, that $\chi(v_{i}) \geq \chi(v_{i+1})$ for all $0 \leq i <k$, and set
$$a_i := 2\cdot i-\chi(v_i).$$
The reader is invited to check that $\mathcal{E}_n=\{\,a_0,\,\dots,\,a_k\,\}$ is indeed the desired initial part of $\mathcal{S}_{C,o}$.

As an example consider the graded root $R$ depicted below. One can see that $\mathfrak{N} = [-8,\infty)\cap \mathbb{Z}$. The edgeless graph $V_{-8}$ is shown on the left with vertices enumerated as described above.\vspace{2mm}

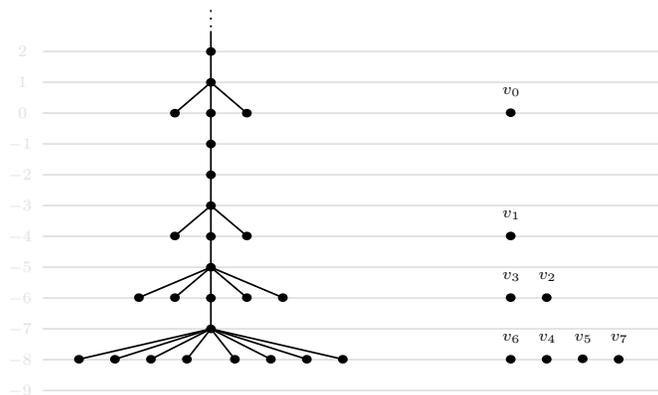
\begin{center}

\tikzset{every picture/.style={line width=0.75pt}} 
\resizebox{9cm}{!}{ 
\tiny
\begin{tikzpicture}[x=0.75pt,y=0.75pt,yscale=-.6,xscale=.7]

\draw [color={rgb, 255:red, 225; green, 225; blue, 225 }  ,draw opacity=1 ][fill={rgb, 255:red, 225; green, 225; blue, 225 }  ,fill opacity=1 ]   (50,200.43) -- (570,200.43) ;
\draw [color={rgb, 255:red, 225; green, 225; blue, 225 }  ,draw opacity=1 ][fill={rgb, 255:red, 225; green, 225; blue, 225 }  ,fill opacity=1 ]   (50,320.43) -- (570,320.43) ;
\draw [color={rgb, 255:red, 225; green, 225; blue, 225 }  ,draw opacity=1 ][fill={rgb, 255:red, 225; green, 225; blue, 225 }  ,fill opacity=1 ]   (50,290.43) -- (570,290.43) ;
\draw [color={rgb, 255:red, 225; green, 225; blue, 225 }  ,draw opacity=1 ][fill={rgb, 255:red, 225; green, 225; blue, 225 }  ,fill opacity=1 ][line width=0.75]    (50,170.43) -- (570,170.43) ;
\draw [color={rgb, 255:red, 225; green, 225; blue, 225 }  ,draw opacity=1 ][fill={rgb, 255:red, 225; green, 225; blue, 225 }  ,fill opacity=1 ]   (50,230.43) -- (570,230.43) ;
\draw [color={rgb, 255:red, 225; green, 225; blue, 225 }  ,draw opacity=1 ][fill={rgb, 255:red, 225; green, 225; blue, 225 }  ,fill opacity=1 ]   (50,260.43) -- (570,260.43) ;
\draw [color={rgb, 255:red, 225; green, 225; blue, 225 }  ,draw opacity=1 ]   (50,50.43) -- (570,50.43) ;
\draw [color={rgb, 255:red, 225; green, 225; blue, 225 }  ,draw opacity=1 ][fill={rgb, 255:red, 225; green, 225; blue, 225 }  ,fill opacity=1 ][line width=0.75]    (50,80.43) -- (570,80.43) ;
\draw [color={rgb, 255:red, 225; green, 225; blue, 225 }  ,draw opacity=1 ][fill={rgb, 255:red, 225; green, 225; blue, 225 }  ,fill opacity=1 ][line width=0.75]    (50,110.43) -- (570,110.43) ;
\draw [color={rgb, 255:red, 225; green, 225; blue, 225 }  ,draw opacity=1 ][fill={rgb, 255:red, 225; green, 225; blue, 225 }  ,fill opacity=1 ][line width=0.75]    (50,140.43) -- (570,140.43) ;
\draw [color={rgb, 255:red, 225; green, 225; blue, 225 }  ,draw opacity=1 ]   (50,380.43) -- (570,380.43) ;
\draw [color={rgb, 255:red, 225; green, 225; blue, 225 }  ,draw opacity=1 ][fill={rgb, 255:red, 225; green, 225; blue, 225 }  ,fill opacity=1 ]   (50,350.43) -- (570,350.43) ;
\draw    (80,350) -- (190,320.43) ;
\draw [shift={(80,350)}, rotate = 344.96] [color={rgb, 255:red, 0; green, 0; blue, 0 }  ][fill={rgb, 255:red, 0; green, 0; blue, 0 }  ][line width=0.75]      (0, 0) circle [x radius= 3.35, y radius= 3.35]   ;
\draw    (110,350) -- (190,320.43) ;
\draw [shift={(110,350)}, rotate = 339.72] [color={rgb, 255:red, 0; green, 0; blue, 0 }  ][fill={rgb, 255:red, 0; green, 0; blue, 0 }  ][line width=0.75]      (0, 0) circle [x radius= 3.35, y radius= 3.35]   ;
\draw    (140,350) -- (190,320.43) ;
\draw [shift={(140,350)}, rotate = 329.4] [color={rgb, 255:red, 0; green, 0; blue, 0 }  ][fill={rgb, 255:red, 0; green, 0; blue, 0 }  ][line width=0.75]      (0, 0) circle [x radius= 3.35, y radius= 3.35]   ;
\draw    (170,350) -- (190,320.43) ;
\draw [shift={(170,350)}, rotate = 304.08] [color={rgb, 255:red, 0; green, 0; blue, 0 }  ][fill={rgb, 255:red, 0; green, 0; blue, 0 }  ][line width=0.75]      (0, 0) circle [x radius= 3.35, y radius= 3.35]   ;
\draw    (210,350) -- (190,320.43) ;
\draw [shift={(210,350)}, rotate = 235.92] [color={rgb, 255:red, 0; green, 0; blue, 0 }  ][fill={rgb, 255:red, 0; green, 0; blue, 0 }  ][line width=0.75]      (0, 0) circle [x radius= 3.35, y radius= 3.35]   ;
\draw    (240,350) -- (190,320.43) ;
\draw [shift={(240,350)}, rotate = 210.6] [color={rgb, 255:red, 0; green, 0; blue, 0 }  ][fill={rgb, 255:red, 0; green, 0; blue, 0 }  ][line width=0.75]      (0, 0) circle [x radius= 3.35, y radius= 3.35]   ;
\draw    (270,350) -- (190,320.43) ;
\draw [shift={(270,350)}, rotate = 200.28] [color={rgb, 255:red, 0; green, 0; blue, 0 }  ][fill={rgb, 255:red, 0; green, 0; blue, 0 }  ][line width=0.75]      (0, 0) circle [x radius= 3.35, y radius= 3.35]   ;
\draw    (300,350) -- (190,320.43) ;
\draw [shift={(300,350)}, rotate = 195.04] [color={rgb, 255:red, 0; green, 0; blue, 0 }  ][fill={rgb, 255:red, 0; green, 0; blue, 0 }  ][line width=0.75]      (0, 0) circle [x radius= 3.35, y radius= 3.35]   ;
\draw    (190,320.43) -- (190,290.43) ;
\draw [shift={(190,320.43)}, rotate = 270] [color={rgb, 255:red, 0; green, 0; blue, 0 }  ][fill={rgb, 255:red, 0; green, 0; blue, 0 }  ][line width=0.75]      (0, 0) circle [x radius= 3.35, y radius= 3.35]   ;
\draw    (190,290.43) -- (190,260.43) ;
\draw [shift={(190,290.43)}, rotate = 270] [color={rgb, 255:red, 0; green, 0; blue, 0 }  ][fill={rgb, 255:red, 0; green, 0; blue, 0 }  ][line width=0.75]      (0, 0) circle [x radius= 3.35, y radius= 3.35]   ;
\draw    (160,290) -- (190,260.43) ;
\draw [shift={(160,290)}, rotate = 315.42] [color={rgb, 255:red, 0; green, 0; blue, 0 }  ][fill={rgb, 255:red, 0; green, 0; blue, 0 }  ][line width=0.75]      (0, 0) circle [x radius= 3.35, y radius= 3.35]   ;
\draw    (130,290) -- (190,260.43) ;
\draw [shift={(130,290)}, rotate = 333.77] [color={rgb, 255:red, 0; green, 0; blue, 0 }  ][fill={rgb, 255:red, 0; green, 0; blue, 0 }  ][line width=0.75]      (0, 0) circle [x radius= 3.35, y radius= 3.35]   ;
\draw    (220,290) -- (190,260.43) ;
\draw [shift={(220,290)}, rotate = 224.58] [color={rgb, 255:red, 0; green, 0; blue, 0 }  ][fill={rgb, 255:red, 0; green, 0; blue, 0 }  ][line width=0.75]      (0, 0) circle [x radius= 3.35, y radius= 3.35]   ;
\draw    (250,290) -- (190,260.43) ;
\draw [shift={(250,290)}, rotate = 206.23] [color={rgb, 255:red, 0; green, 0; blue, 0 }  ][fill={rgb, 255:red, 0; green, 0; blue, 0 }  ][line width=0.75]      (0, 0) circle [x radius= 3.35, y radius= 3.35]   ;
\draw    (190,260.43) -- (190,230.43) ;
\draw [shift={(190,260.43)}, rotate = 270] [color={rgb, 255:red, 0; green, 0; blue, 0 }  ][fill={rgb, 255:red, 0; green, 0; blue, 0 }  ][line width=0.75]      (0, 0) circle [x radius= 3.35, y radius= 3.35]   ;
\draw    (160,230) -- (190,200.43) ;
\draw [shift={(160,230)}, rotate = 315.42] [color={rgb, 255:red, 0; green, 0; blue, 0 }  ][fill={rgb, 255:red, 0; green, 0; blue, 0 }  ][line width=0.75]      (0, 0) circle [x radius= 3.35, y radius= 3.35]   ;
\draw    (220,230) -- (190,200.43) ;
\draw [shift={(220,230)}, rotate = 224.58] [color={rgb, 255:red, 0; green, 0; blue, 0 }  ][fill={rgb, 255:red, 0; green, 0; blue, 0 }  ][line width=0.75]      (0, 0) circle [x radius= 3.35, y radius= 3.35]   ;
\draw    (190,230.43) -- (190,200.43) ;
\draw [shift={(190,230.43)}, rotate = 270] [color={rgb, 255:red, 0; green, 0; blue, 0 }  ][fill={rgb, 255:red, 0; green, 0; blue, 0 }  ][line width=0.75]      (0, 0) circle [x radius= 3.35, y radius= 3.35]   ;
\draw    (190,200.43) -- (190,170.43) ;
\draw [shift={(190,200.43)}, rotate = 270] [color={rgb, 255:red, 0; green, 0; blue, 0 }  ][fill={rgb, 255:red, 0; green, 0; blue, 0 }  ][line width=0.75]      (0, 0) circle [x radius= 3.35, y radius= 3.35]   ;
\draw    (190,170.43) -- (190,140.43) ;
\draw [shift={(190,170.43)}, rotate = 270] [color={rgb, 255:red, 0; green, 0; blue, 0 }  ][fill={rgb, 255:red, 0; green, 0; blue, 0 }  ][line width=0.75]      (0, 0) circle [x radius= 3.35, y radius= 3.35]   ;
\draw    (190,140.43) -- (190,110.43) ;
\draw [shift={(190,140.43)}, rotate = 270] [color={rgb, 255:red, 0; green, 0; blue, 0 }  ][fill={rgb, 255:red, 0; green, 0; blue, 0 }  ][line width=0.75]      (0, 0) circle [x radius= 3.35, y radius= 3.35]   ;
\draw    (160,110.43) -- (190,80.43) ;
\draw [shift={(160,110.43)}, rotate = 315] [color={rgb, 255:red, 0; green, 0; blue, 0 }  ][fill={rgb, 255:red, 0; green, 0; blue, 0 }  ][line width=0.75]      (0, 0) circle [x radius= 3.35, y radius= 3.35]   ;
\draw    (190,110.43) -- (190,80.43) ;
\draw [shift={(190,110.43)}, rotate = 270] [color={rgb, 255:red, 0; green, 0; blue, 0 }  ][fill={rgb, 255:red, 0; green, 0; blue, 0 }  ][line width=0.75]      (0, 0) circle [x radius= 3.35, y radius= 3.35]   ;
\draw    (220,110.43) -- (190,80.43) ;
\draw [shift={(220,110.43)}, rotate = 225] [color={rgb, 255:red, 0; green, 0; blue, 0 }  ][fill={rgb, 255:red, 0; green, 0; blue, 0 }  ][line width=0.75]      (0, 0) circle [x radius= 3.35, y radius= 3.35]   ;
\draw    (190,80.43) -- (190,50.43) ;
\draw [shift={(190,80.43)}, rotate = 270] [color={rgb, 255:red, 0; green, 0; blue, 0 }  ][fill={rgb, 255:red, 0; green, 0; blue, 0 }  ][line width=0.75]      (0, 0) circle [x radius= 3.35, y radius= 3.35]   ;
\draw    (190,50.43) -- (190,30.43) ;
\draw [shift={(190,50.43)}, rotate = 270] [color={rgb, 255:red, 0; green, 0; blue, 0 }  ][fill={rgb, 255:red, 0; green, 0; blue, 0 }  ][line width=0.75]      (0, 0) circle [x radius= 3.35, y radius= 3.35]   ;
\draw  [dash pattern={on 0.84pt off 2.51pt}]  (190,10) -- (190,30.43) ;
\draw    (440,110) ;
\draw [shift={(440,110)}, rotate = 0] [color={rgb, 255:red, 0; green, 0; blue, 0 }  ][fill={rgb, 255:red, 0; green, 0; blue, 0 }  ][line width=0.75]      (0, 0) circle [x radius= 3.35, y radius= 3.35]   ;
\draw    (440,230) ;
\draw [shift={(440,230)}, rotate = 0] [color={rgb, 255:red, 0; green, 0; blue, 0 }  ][fill={rgb, 255:red, 0; green, 0; blue, 0 }  ][line width=0.75]      (0, 0) circle [x radius= 3.35, y radius= 3.35]   ;
\draw    (470,290) ;
\draw [shift={(470,290)}, rotate = 0] [color={rgb, 255:red, 0; green, 0; blue, 0 }  ][fill={rgb, 255:red, 0; green, 0; blue, 0 }  ][line width=0.75]      (0, 0) circle [x radius= 3.35, y radius= 3.35]   ;
\draw    (440,290) ;
\draw [shift={(440,290)}, rotate = 0] [color={rgb, 255:red, 0; green, 0; blue, 0 }  ][fill={rgb, 255:red, 0; green, 0; blue, 0 }  ][line width=0.75]      (0, 0) circle [x radius= 3.35, y radius= 3.35]   ;
\draw    (500,349.57) ;
\draw [shift={(500,349.57)}, rotate = 0] [color={rgb, 255:red, 0; green, 0; blue, 0 }  ][fill={rgb, 255:red, 0; green, 0; blue, 0 }  ][line width=0.75]      (0, 0) circle [x radius= 3.35, y radius= 3.35]   ;
\draw    (440,350) ;
\draw [shift={(440,350)}, rotate = 0] [color={rgb, 255:red, 0; green, 0; blue, 0 }  ][fill={rgb, 255:red, 0; green, 0; blue, 0 }  ][line width=0.75]      (0, 0) circle [x radius= 3.35, y radius= 3.35]   ;
\draw    (530,350) ;
\draw [shift={(530,350)}, rotate = 0] [color={rgb, 255:red, 0; green, 0; blue, 0 }  ][fill={rgb, 255:red, 0; green, 0; blue, 0 }  ][line width=0.75]      (0, 0) circle [x radius= 3.35, y radius= 3.35]   ;
\draw    (470,350) -- (470,350.43) ;
\draw [shift={(470,350)}, rotate = 90] [color={rgb, 255:red, 0; green, 0; blue, 0 }  ][fill={rgb, 255:red, 0; green, 0; blue, 0 }  ][line width=0.75]      (0, 0) circle [x radius= 3.35, y radius= 3.35]   ;

\draw (27,42.4) node [anchor=north west][inner sep=0.75pt]  [color={rgb, 255:red, 225; green, 225; blue, 225 }  ,opacity=1 ]  {$\textcolor[rgb]{0.88,0.88,0.88}{2}$};
\draw (27,72.4) node [anchor=north west][inner sep=0.75pt]  [color={rgb, 255:red, 225; green, 225; blue, 225 }  ,opacity=1 ]  {$1$};
\draw (27,102.4) node [anchor=north west][inner sep=0.75pt]  [color={rgb, 255:red, 225; green, 225; blue, 225 }  ,opacity=1 ]  {$0$};
\draw (19,132.4) node [anchor=north west][inner sep=0.75pt]  [color={rgb, 255:red, 225; green, 225; blue, 225 }  ,opacity=1 ]  {$-1$};
\draw (19,162.4) node [anchor=north west][inner sep=0.75pt]  [color={rgb, 255:red, 225; green, 225; blue, 225 }  ,opacity=1 ]  {$-2$};
\draw (19,192.4) node [anchor=north west][inner sep=0.75pt]  [color={rgb, 255:red, 225; green, 225; blue, 225 }  ,opacity=1 ]  {$-3$};
\draw (19,222.4) node [anchor=north west][inner sep=0.75pt]  [color={rgb, 255:red, 225; green, 225; blue, 225 }  ,opacity=1 ]  {$-4$};
\draw (19,252.4) node [anchor=north west][inner sep=0.75pt]  [color={rgb, 255:red, 225; green, 225; blue, 225 }  ,opacity=1 ]  {$-5$};
\draw (19,282.4) node [anchor=north west][inner sep=0.75pt]  [color={rgb, 255:red, 225; green, 225; blue, 225 }  ,opacity=1 ]  {$-6$};
\draw (19,312.4) node [anchor=north west][inner sep=0.75pt]  [color={rgb, 255:red, 225; green, 225; blue, 225 }  ,opacity=1 ]  {$-7$};
\draw (19,342.4) node [anchor=north west][inner sep=0.75pt]  [color={rgb, 255:red, 225; green, 225; blue, 225 }  ,opacity=1 ]  {$-8$};
\draw (19,372.4) node [anchor=north west][inner sep=0.75pt]  [color={rgb, 255:red, 225; green, 225; blue, 225 }  ,opacity=1 ]  {$-9$};
\draw (431,82.4) node [anchor=north west][inner sep=0.75pt]    {$v_{0}$};
\draw (431,202.4) node [anchor=north west][inner sep=0.75pt]    {$v_{1}$};
\draw (461,262.4) node [anchor=north west][inner sep=0.75pt]    {$v_{2}$};
\draw (431,262.4) node [anchor=north west][inner sep=0.75pt]    {$v_{3}$};
\draw (461,322.4) node [anchor=north west][inner sep=0.75pt]    {$v_{4}$};
\draw (491,322.4) node [anchor=north west][inner sep=0.75pt]    {$v_{5}$};
\draw (431,322.4) node [anchor=north west][inner sep=0.75pt]    {$v_{6}$};
\draw (521,322.4) node [anchor=north west][inner sep=0.75pt]    {$v_{7}$};

\end{tikzpicture}
}

\end{center}\vspace{1mm}
\noindent One now reads off that $\mathcal{E}_{-8} = \{\,0,\,6,\,10,\,12,\,16,\,18,\,20,\,22 \,\}$. The graded root in the image is in fact the graded root associated to the numerical semigroup minimally generated by the set $\mathcal{S}=\{\,6,\,10,\,31\,\}$ and one checks that $\mathcal{E}_{-8} = \mathcal{S} \cap [0,23]$ indeed (here $23=\delta$).

\subsection{`$e$ comes late enough'}\label{subs:elate}\,

In this subsection we formalize the statement that $e=\min \mathfrak{N}$ is small enough in the sense, that $\mathcal{E}$ describes a large enough part of semigroup to be then able to recover it fully, just by knowing the $\delta$ invariant.

First we show, that in the special case of $l_{g-1}=2$ (see the notations in item \textit{(v)} of Proposition \ref{rem:prelim}), the initial part contains every semigroup element until $\delta$.

\begin{prop}\label{prop:lg-1=2}
    Let $(C, o) \in (\mathbb{C}^2, o)$ be an irreducible plane curve singularity with  $l_{g-1} = 2$. Then $\mathcal{E}= [0, \delta]\cap \mathcal{S}_{C, o}$ and $\bH^0_{\geq 2e}$ is closed under the $U$-action with $\bH^0_{\geq 2e} = \bH^0$ as graded $\bZ[U]$-modules.
\end{prop}

First we need the following lemma.

\begin{lemma}\label{lem:w(l)inN}
    Let $\ell \in \mathbb{Z}_{\geq0}$ be a lattice point, such that there are no consecutive elements between $0$ and $\ell$ in some numerical semigroup $\mathcal{S}$ (i.e. for any $0\leq \ell' < \ell, \ \ell' \in \mathcal{S}$ we have $\ell'+1 \notin \mathcal{S}$) and $\max \big( w \big|_{[ \ell, \delta]}\big) = w(\ell)$, where $w$ is computed from $\mathcal{S}$ via {\rm (\ref{eq:wfromS})}. Then $w(\ell) \in \mathfrak{N}$ (hence $w(\ell) \geq e$).
\end{lemma}

\begin{proof}
    The fact that there are no consecutive semigroup elements before $\ell$ implies, that for any $0 \leq \ell' \leq \ell-2$ we have $w(\ell') \geq w(\ell'+2)$ (cf. formula  (\ref{eq:wfromS})). Therefore, for any $n \geq w(\ell)$, the following are true for the space $S_n$: firstly, $[\ell, \delta] \subset S_n$, and secondly,
\begin{equation} \label{eq:linS->l+2inS}
\text{for any } 0 \leq \ell' \leq \ell: \text{ if } \ell' \in S_n, \text{ then also } \ell'+2 \in S_n. 
\end{equation}
Thus, if we denote by $s_n$ the smallest lattice point of $S_n$, then $s_n + 2 \mathbb{N} \subset S_n$ and if $s_n + 2k_n +1$ denotes the smallest element of $S_n \setminus \big( s_n + 2 \mathbb{N}\big)$, then $[s_n +2k_n, \delta] \subset S_n$. Therefore, for any $n \geq w(\ell)$, $S_n$ is of form prescribed by Proposition \ref{prop:ninNhaSn}, and therefore $n \in \mathfrak{N}$ for every $n \geq w(\ell)$. 
\end{proof}

\begin{proof}[Proof of Proposition \ref{prop:lg-1=2}.]
Firstly, if $g=1$ (i.e. we have a single Puiseux pair) and $m= l_{g-1}=2$, then we already saw in the proof of Proposition \ref{prop:0inN}, that $0=\min w_0 \in \mathfrak{N}$ and hence $\mathbb{H}_{\geq 0}=\mathbb{H}$. 

Let us now suppose, that $g >1$. We claim, that it is sufficient to prove that $\delta < \beta_g$ as this implies that $\mathcal{E}=[0,\delta]\cap\mathcal{S}_{C,o}$. 

    Indeed, by inequality (\ref{eq:nibetai<}) there can be no consecutive elements in $\mathcal{S}_{C, o}\cap[0,\beta_{g}-1]$, so, if $\delta < \beta_g$, then we can apply the previous Lemma \ref{lem:w(l)inN} to $\ell=\delta$ and we obtain that $w(\delta) \in \mathfrak{N}$. Note however, that $\min w_0$ is obtained by some lattice point in $[0, \delta]$, hence, by property (\ref{eq:linS->l+2inS}) $w(\delta) \leq \min w_0 +1$. Considering also the fact, that $U \cdot \mathbb{H}^0_{2\min w_0}=0$, we obtain that $\min w_0 \in \mathfrak{N}$, hence $\mathcal{E}=[0, \delta] \cap \mathcal{S}_{C,o}$ and $\mathbb{H}^0_{\geq 2e}=\mathbb{H}^0_{\geq 2\min w_0}=\mathbb{H}^0$.

    In order to prove the inequality $\delta < \beta_g$, we define the partial conductors
    $$c_i := \min\big\{\, n \in \bN \ \big| \ n+ \ell_i\bN \subseteq \langle \beta_0,\dots,\beta_i \rangle \,\big\}.$$
    Note that $c_0=0$, while $c_g=c$ is just the conductor of $\mathcal{S}_{C, o}$. It follows from \textit{(v)} of Proposition \ref{rem:prelim} that
    \begin{equation}\label{eq:ci}
        c_i = c_{i-1} + \frac{(l_{i-1}-l_i)(\beta_i-l_i)}{l_i},
    \end{equation}
    see e.g. \cite[Corollary 4.3]{bdr03}. Employing equation (\ref{eq:ci}) and inequality (\ref{eq:nibetai<}), one can show inductively that $\beta_{i+1} > c_i$ for all $0 \leq i \leq g-1$.
    Finally, using $c_{g-1}<\beta_g$, $l_{g-1}=2$ and equation (\ref{eq:ci}) once again, we get that
    \begin{equation} \label{eq:delta<beta_g}
        \delta = \frac{c_g}{2} =\frac{c_{g-1}+\beta_g-1}{2} < \beta_g
    \end{equation}
    as claimed.
\end{proof}

\begin{corollary}\label{cor:beta_g-1}
    Let us suppose that $\mathcal{S}_{C, o} \neq \langle 2, 3\rangle$ is a semigroup of an irreducible plane curve singularity with $l_{g-1}=2$. Then $\beta_{g-1} \in \mathcal{E}$. 
\end{corollary}

\begin{proof}
    If $g=1$, then $\mathcal{S}_{C, o}$ is of type $\langle 2, 2k+1\rangle$ for some $k>1$. In this case $\beta_0=2 \in \mathcal{E}$ trivially holds, where $\mathcal{E} = [0, k] \cap \mathcal{S}$ by Proposition \ref{prop:lg-1=2}.

    If $g \geq 2$, then clearly $2 \notin \mathcal{S}_{C, o}$ ($2 \in \mathcal{S}_{C, o}$ would imply $g=1$), hence, since $\beta_g$ is the smallest odd element of $\mathcal{S}_{C, o}$, $\beta_g + 2 \notin \mathcal{S}_{C, o}$ and, thus,  $\beta_g < c$. Now from inequality (\ref{eq:nibetai<})  we get that $\beta_{g-1} < \beta_g/2 < c/2 = \delta$. Therefore, $\beta_{g-1} \in [0, \delta]\cap \mathcal{S}_{C,o}$, which equals $\mathcal{E}$ by Proposition \ref{prop:lg-1=2}.
\end{proof}

Next, we show that this property is also true for any `large enough' semigroup as well.

\begin{prop}\label{prop:genericcase}
    If $g \geq 2$ (i.e. $\mathcal{S}_{C, o}$ is minimally generated by at least three elements), then $\beta_{g-1} \in \mathcal{E}$ (and, since $\mathcal{E}$ is and initial part, $\{ \beta_0, \ldots, \beta_{g-1}\} \subset \mathcal{E}$).
\end{prop}

\noindent Before proving Proposition \ref{prop:genericcase} we will need the following lemma.

\begin{lemma}\label{lem:wbetag-1}
    If $g \geq 2$ and $l_{g-1} \neq 2$, then $w(\beta_{g-1}) = w(\beta_{g-1}+2)= \max\big( w\big|_{[\beta_{g-1}+2,\delta]}\big).$
\end{lemma}

\begin{proof}

    By inequality (\ref{eq:nibetai<})  $\beta_g > 2\beta_{g-1}>\beta_{g-1}+m$, i.e. $\beta_g \geq \beta_{g-1} + m + 2$. It follows that $[\beta_{g-1}+2,\beta_{g-1}+m+1]\cap \mathcal{S}_{C, o} \subseteq l_{g-1}\bZ$. Since $l_{g-1} \neq 2$, $\beta_{g-1}+2\not\in \mathcal{S}_{C, o}$, with $w(\beta_{g-1}+2) = w(\beta_{g-1})$ and also $w\big|_{[\beta_{g-1}+2,\beta_{g-1}+m+1]} \leq w(\beta_{g-1})$. Finally, applying Lemma \ref{lem:ablakos} to $[\beta_{g-1}+2,\beta_{g-1}+m+1]$ finishes the proof.
\end{proof}

\begin{proof}[Proof of Proposition \ref{prop:genericcase}.]
The case $l_{g-1}=2$ was already discussed in Corollary \ref{cor:beta_g-1}. 

Let us then consider the case of $l_{g-1}\neq 2$. Notice, that there can be no consecutive elements in $\mathcal{S}_{C, o}\cap[0,\beta_{g-1}+2]$ by inequality (\ref{eq:nibetai<}). This, together with Lemma \ref{lem:wbetag-1} and \ref{lem:w(l)inN} used for $\ell = \beta_{g-1}+2$  implies that    
    $w(\beta_{g-1}) \in \mathfrak{N}$. It follows that $\beta_{g-1} \in \mathcal{E}$ by definition.
\end{proof}

Via contraposition, Proposition \ref{prop:genericcase} implies, that $\mathcal{E}$ can only be minimal (i.e. $\mathcal{E}=\{ 0\}$, see Proposition \ref{prop:0inN}), if the semigroup $\mathcal{S}_{C, o}$ is small in the sense that it has at most two generators:

\begin{corollary}\label{cor:E=0}
    If $\mathcal{E}=\{ 0\}$, then $g\leq 1$, i.e. $\mathcal{S}_{C, o}=\langle\beta_0, \beta_1\rangle$ with $\beta_0, \beta_1 \in \mathbb{N}_{\geq 1}, \ {\rm gcd}(\beta_0, \beta_1)=1$.
\end{corollary}

We are now ready to put together the pieces of the proof of the main theorem.

\subsection{The algorithm}\label{subs:alg}\,

We will describe here the algorithm how to reconstruct the semigroup $\mathcal{S}_{C, o}$, more specifically its minimal set of generators $\{ \beta_0, \ldots, \beta_g\}$, of an irreducible plane curve singularity $(C, o) \subset (\mathbb{C}^2, o)$ from the $\mathbb{Z}[U]$-module $\mathbb{H}^0(C, o)$.

\begin{enumerate}
    \item First check whether $\mathbb{H}^0$ is trivial or not: if $\mathbb{H}^0 \cong T_{0}^{\infty}$ (or equivalently ${\rm rank}_\mathbb{Z}\mathbb{H}^0_{2n}=1$ for $n \geq 0$ and $0$ otherwise), then $(C, o)$ is smooth.
    \item If $\mathbb{H}^0$ is nontrivial, find $e=\min \mathfrak{N}$, the smallest $n \geq \min w_0$ such that the $U$-action on $\mathbb{M}_{\geq 2n} / T_{2 \min w_0}^\infty $ is trivial (for example, by obtaining the canonical direct sum decomposition of $\mathbb{H}^0$ of Proposition \ref{prop:directsumdecomp} with the method described in \cite[Remark 2.7]{LCandmult} and then reading off $e$ as in (\ref{eq:eadirsumdecompbol})).
    \item Then reconstruct the set $\mathcal{E}=[0, s_e +2k_e] \cap \mathcal{S}_{C, o}
    $ from the $\mathbb{Z}[U]$-module $\mathbb{H}^0$ by the algorithm proceeding Proposition \ref{prop:emeghat}.
    \item If $\mathcal{E}$ is minimal, i.e. $\mathcal{E} =\{0\}$, then $g=1$, or equivalently, the curve has a single Puiseux pair. Thus $\mathcal{S}_{C,o}=\langle \beta_0, \beta_1\rangle$, with $\beta_0=m$ the multiplicity and, hence, $\beta_0=2$ if $\mathbb{H}^0_{<0}=0$, and $\beta_0=2-\max\{n\,|\,  {\rm ker} (U:\bH^0_{2n}\to \bH^0_{2n-2}) \not =0,\ n<0\}$ otherwise. Now $\delta = (\beta_0-1)(\beta_1-1)/2= {\rm rank}_{\bZ}\bH^0_{\leq 0}(C,o)-1$ and from these compute $\beta_1$ as well.
    \item If $\mathcal{E}$ is not minimal, then $\{ \beta_0, \ldots, \beta_{g-1}\} \subset \mathcal{E}$ and they can be read off as $\mathcal{E}^* \setminus (\mathcal{E}^* + \mathcal{E}^*)$ (where $\mathcal{E}^* = \mathcal{E}\setminus \{0\}$), i.e. the primitive elements of $\mathcal{E}$. Also
    \begin{equation*}
        {\rm rank}_{\bZ}\bH^0_{\leq 0}(C,o)-1 = \delta = c/2 = c_{g-1} + {(l_{g-1}-1)(\beta_g-1)},
    \end{equation*}
    where the partial conductor $c_{g-1}$ can be computed inductively from the set $\{ \beta_0, \ldots, \beta_{g-1}\}$ via the formula (\ref{eq:ci}) and $l_{g-1}={\rm gcd}(\beta_0, \ldots, \beta_{g-1})$. From these compute $\beta_g$ as well.
\end{enumerate}

\begin{proof}[Proof of Theorem \ref{thm:MAIN} / previous algorithm]
For the first step we use \cite[Example 4.6.1]{AgNeCurve} (see also part \textit{(vii)} of Proposition \ref{rem:prelim}). The second step is straightforward, while the third is just the application of Propositions \ref{prop:kezdoszelet} and \ref{prop:emeghat}.

The main assertion of step (4) regarding the number of Puiseux pairs follows from Corollary \ref{cor:E=0}. Then the multiplicity and delta invariant are read off as described in parts \textit{(vii)} and \textit{(viii)} of Proposition \ref{rem:prelim}.

In order to prove step (5) we have to consider two different cases: 
\begin{itemize}
    \item if $g=1$, then $\mathcal{E} \supsetneq \{0\}$ implies that $\beta_0 \in \mathcal{E}$, since $\mathcal{E}$ is an initial part of $\mathcal{S}_{C, o}$ (see Proposition \ref{prop:kezdoszelet});
    \item if $g\geq 2$, then $\{ \beta_0, \ldots, \beta_{g-1}\} \subset \mathcal{E}$ by Proposition \ref{prop:genericcase}.
\end{itemize}
($g=0$ would mean that $(C, o)$ is smooth, which case was dealt with in step (1)). Finally, the $\delta$ invariant is read off again as described in part \textit{(viii)} of Proposition \ref{rem:prelim}.
\end{proof}

\subsection{$l_{g-1}=2$ revisited}\,

We already saw in Proposition \ref{prop:lg-1=2}, that the condition $l_{g-1}=2$ is special. Here we prove that we can, in fact, characterize it with the $\mathbb{Z}[U]$-module $\mathbb{H}^0(C, o)$ as well.

\begin{prop}
    The following conditions are equivalent for a non-smooth irreducible plane curve singularity $(C, o) \subset (\mathbb{C}^2, o)$:
\begin{enumerate}
    \item $l_{g-1}=2$;
    \item $\min w_0$ even and ${\rm rank}_{\mathbb{Z}}\big( \mathbb{H}^0_{4k+2}\big)=1$ for all $k \geq ({\min w_0})/{2}$. This clearly implies that $U \cdot \mathbb{H}^0 \leq T_{2 \min w_0}^{\infty}$ and that, if $\alpha \in \mathbb{H}^0$ is homogeneous and $U \cdot \alpha=0$, then $\alpha$ has even degree (i.e. the homogeneous parts of the kernel of the $U$-action are of even degree).
     \item In the graded root $R(C,o)$ the gradings of the local minimum points are all even and the trimmed version $\overline{R}$ only has a single leaf. 
\end{enumerate}
\end{prop} 

\begin{proof} We will use the correspondence (\ref{eq:locmin}) between local minimum points of the weight function and the kernel of the $U$-action. In this language, property \textit{(2)} corresponds to the facts that $S_{2k+1}$ is connected for any $k \geq (\min w_0)/2$, and, for any local minimum point $\ell$, the local minimum value $w(\ell)$ is even. Notice that the parity of $\ell \in \mathbb{Z}_{\geq0}$ and $w(\ell) \in \mathbb{Z}$ always agree, hence this second condition is equivalent to every local minimum point being even. 

\noindent \textit{(1)} $\Rightarrow$ \textit{(2)}. If $l_{g-1}=2$, we already saw in (\ref{eq:delta<beta_g}) that $ \beta_{g} > \delta$ (more specifically, there we only considered the case of $g >1$, however, for $g=1$ we have $\beta_1 > c > \delta$ trivially). Thus until $\delta$ all the local minima of the weight function (or, equivalently, all the semigroup values) are even, thus their analytic weight $w(\ell)=2\mathfrak{h}(\ell) - \ell$ is also even, so is then the grading on the homogeneous parts of the kernel of the $U$-action. On the other hand, Proposition \ref{prop:lg-1=2} implies that $\min w_0 \in \mathfrak{N}$ and hence, by Proposition \ref{prop:ninNhaSn},  $S_{2k+1}$ is connected for any $k \geq (\min w_0)/2$.

\noindent \textit{(2)} $\Rightarrow$ \textit{(1)}. First we consider the case of $\beta_g > \delta$. Then by (\ref{eq:ci})
\begin{align*}
    2\beta_g > c=c_g=c_{g-1}+(\beta_g-1)(l_{g-1}-1) \Leftrightarrow \  &\\
    \Leftrightarrow 3\beta_g-1-c_{g-1} >\ & l_{g-1}(\beta_g-1)\Leftrightarrow \\
    \Leftrightarrow 3+\frac{2-c_{g-1}}{\beta_{g}-1} >\  &l_{g-1},
\end{align*}
so $ l_{g-1}$ must be $2$, except when $c_{g-1}\leq1$. But if $g > 1$, then $c_{g-1} >0$ and $l_{g-1}|c_{g-1}$, so $l_{g-1}\neq 2$ can only happen if $g=1$, i.e. we only have a single Puiseux pair. Then $2\leq l_{g-1}=\beta_0 < \beta_1 = \beta_g$ and $c_0=0$, so the only possible cases with $l_{g-1} >2$  are: $\beta_0=3, \beta_1 >3$ such that ${\rm gcd}(3,\beta_1)=1$. However, in these cases the multiplicity is $3$, which gives a local minimum point with weight $-1$ odd (see part \textit{(vii)} of Proposition \ref{rem:prelim}). 

We can turn to the $\beta_g \leq \delta$ case.
We prove that this case is in fact vacuous: a semigroup $\mathcal{S}_{C, o}$ of an irreducible curve singularity with $\beta_g < \delta$ cannot have lattice cohomology module $\mathbb{H}$ satisfying the conditions of property \textit{(2)}. Indeed, every semigroup element $s \in \mathcal{S}_{C, o}$, with $s-1 \notin \mathcal{S}_{C, o}$ is a local minimum point (see Definition \ref{def:locmin}) and hence must be even. As there are no consecutive semigroup elements between $0$ and $\beta_g-1$ (see part \textit{(v)} of Proposition \ref{rem:prelim}) the minimal generators $\beta_0, \beta_1, \ldots, \beta_{g-1}$ must all be even. Since ${\rm gcd}(\beta_0, \ldots, \beta_{g-1}, \beta_g)=1$, $\beta_g$ must be odd and thus, by the previous local minimum point argument, it must have $\beta_g-1 \in \mathcal{S}_{C, o}$. Now by formula (\ref{eq:w_0valtozasa}) we have $w(\beta_g-1) < w(\beta_g)< w(\beta_g+1)$  and, thus, $\beta_g \neq \delta$, since that would contradict the Gorenstein symmetry (\ref{eq:Gorsym}).  But then $S_{w(\beta_g)}$ cannot be connected (with $w(\beta_g) \equiv \beta_g \equiv 1\ \mod2$), since $\{\beta_g, c-\beta_g\} \in S_{w(\beta_g)}$ and $\beta_{g}+1 \notin S_{w(\beta_g)}$. So property \textit{(2)} cannot be satisfied if $\beta_g \leq \delta$.

The equivalence \textit{(2)} $\Leftrightarrow$ \textit{(3)} is clear through correspondence (\ref{eq:locmin}) between the kernel of the $U$-action and the local minimum points of the graded root.
\end{proof}

\begin{corollary}
    The proof in fact shows that $l_{g-1} = 2$ is also equivalent to the following: $\beta_g > \delta$ and the multiplicity is not $3$.
\end{corollary}

\section{Lattice cohomology and the Seifert form}\label{subs:seifert}
\subsection{}
In this section we show that  in general the Seifert form $S(C,o)$ and the lattice cohomology module $\bH^*(C,o)$ of a plane curve germ $(C,o) \subset (\bC^2,o)$ do not determine each other. (For the definition of the Seifert form see e.g. \cite[Chapter 10]{wall}.) Note that for irreducible plane curves both are complete embedded topological invariants, so our examples will all be reducible.

First, we present an example of two singularities having the same Seifert form but distinct lattice cohomology modules. Consider the germs $(C_{r,s},o) \subset (\bC^2,o)$ for some $r, s \in \mathbb{N}$ given by
$$f_{r,s}(x,y) = \big( (x^3-y^2)^2-x^{s+6}-4x^{(s+9)/2}y\big)\big( (x^2-y^5)^2-y^{r+10}-4xy^{(r+15)/2}\big)=0.$$
It was shown in \cite{duboismichel} that if $r$ and $s$ are odd integers larger then 10, then the germs $C_{r,s}$ and $C_{s-8,r+8}$ have isomorphic integral Seifert forms.

Consider the plane curve germs $(C_{11,11},o)$ and $(C_{3,19},o)$ for instance. According to the previous paragraph their integral Seifert forms are isomorphic, whereas their lattice cohomology modules are distinct as can be seen on the image below. It shows their respective graded roots (corresponding to $\bH^0$) and a number of circles representing $\bH^1$. Clearly, $\bH^*(C_{11,11},o) \not\cong \bH^*(C_{3,19},o)$.
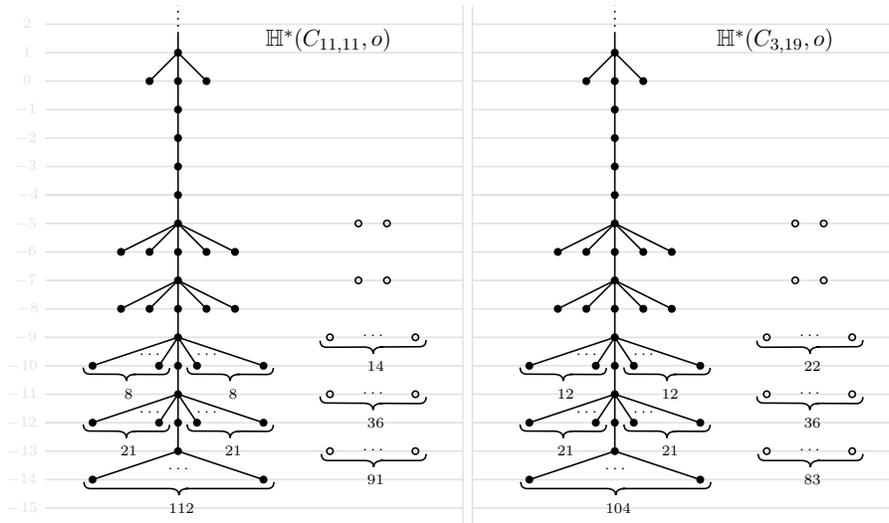
\vspace{2mm}
\begin{center}

\tiny
\tikzset{every picture/.style={line width=0.75pt}} 
\resizebox{12cm}{!}{
}

\end{center}\vspace{2mm}

Next we give an example of two singularities having isomorphic lattice cohomology modules but distinct Seifert forms. Let $n \geq 1$ be an integer and consider the germs $(C,o)$ and $(C',o)$ given by $(x^n+y^{n+1})(x^{n+1}+y^n)=0$ and $(x^{n-1}+y^n)(x^{n+2}+y^{n+1})=0$ respectively.

 The intersection multiplicity of the two branches of $(C,o)$ is $n^2$, whereas in the case of $(C',o)$ it is $n^2-1$. Since the intersection multiplicities of different branches are determined by the integral Seifert form according to \cite{kaenders}, this shows that $(C,o)$ and $(C',o)$ have distinct Seifert forms. On the other hand, we claim that $R(C,o) \cong R(C',o)$ for all $n$. As an example, consider the graded root $R(C,o) \cong R(C',o)$ depicted below in the case $n=4$.
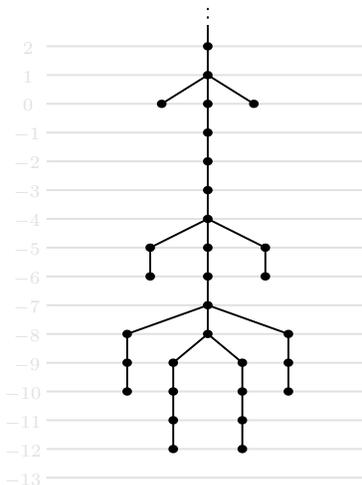
\begin{center}

\tiny
\tikzset{every picture/.style={line width=0.75pt}} 
\resizebox{5cm}{!}{
\begin{tikzpicture}[x=0.75pt,y=0.75pt,yscale=-.5,xscale=.6]

\draw [color={rgb, 255:red, 225; green, 225; blue, 225 }  ,draw opacity=1 ][fill={rgb, 255:red, 225; green, 225; blue, 225 }  ,fill opacity=1 ]   (90,230) -- (370,230) ;
\draw [color={rgb, 255:red, 225; green, 225; blue, 225 }  ,draw opacity=1 ][fill={rgb, 255:red, 225; green, 225; blue, 225 }  ,fill opacity=1 ]   (90,470) -- (370,470) ;
\draw [color={rgb, 255:red, 225; green, 225; blue, 225 }  ,draw opacity=1 ][fill={rgb, 255:red, 225; green, 225; blue, 225 }  ,fill opacity=1 ]   (90,440) -- (370,440) ;
\draw [color={rgb, 255:red, 225; green, 225; blue, 225 }  ,draw opacity=1 ][fill={rgb, 255:red, 225; green, 225; blue, 225 }  ,fill opacity=1 ]   (90,350) -- (370,350) ;
\draw [color={rgb, 255:red, 225; green, 225; blue, 225 }  ,draw opacity=1 ][fill={rgb, 255:red, 225; green, 225; blue, 225 }  ,fill opacity=1 ]   (90,380) -- (370,380) ;
\draw [color={rgb, 255:red, 225; green, 225; blue, 225 }  ,draw opacity=1 ][fill={rgb, 255:red, 225; green, 225; blue, 225 }  ,fill opacity=1 ]   (90,410) -- (370,410) ;
\draw [color={rgb, 255:red, 225; green, 225; blue, 225 }  ,draw opacity=1 ][fill={rgb, 255:red, 225; green, 225; blue, 225 }  ,fill opacity=1 ]   (90,500) -- (370,500) ;
\draw [color={rgb, 255:red, 225; green, 225; blue, 225 }  ,draw opacity=1 ][fill={rgb, 255:red, 225; green, 225; blue, 225 }  ,fill opacity=1 ]   (90,320) -- (370,320) ;
\draw [color={rgb, 255:red, 225; green, 225; blue, 225 }  ,draw opacity=1 ][fill={rgb, 255:red, 225; green, 225; blue, 225 }  ,fill opacity=1 ]   (90,200) -- (370,200) ;
\draw [color={rgb, 255:red, 225; green, 225; blue, 225 }  ,draw opacity=1 ][fill={rgb, 255:red, 225; green, 225; blue, 225 }  ,fill opacity=1 ]   (90,260) -- (370,260) ;
\draw [color={rgb, 255:red, 225; green, 225; blue, 225 }  ,draw opacity=1 ][fill={rgb, 255:red, 225; green, 225; blue, 225 }  ,fill opacity=1 ]   (90,290) -- (370,290) ;
\draw [color={rgb, 255:red, 225; green, 225; blue, 225 }  ,draw opacity=1 ]   (90,80) -- (370,80) ;
\draw [color={rgb, 255:red, 225; green, 225; blue, 225 }  ,draw opacity=1 ][fill={rgb, 255:red, 225; green, 225; blue, 225 }  ,fill opacity=1 ]   (90,110) -- (370,110) ;
\draw [color={rgb, 255:red, 225; green, 225; blue, 225 }  ,draw opacity=1 ][fill={rgb, 255:red, 225; green, 225; blue, 225 }  ,fill opacity=1 ]   (90,140) -- (370,140) ;
\draw [color={rgb, 255:red, 225; green, 225; blue, 225 }  ,draw opacity=1 ][fill={rgb, 255:red, 225; green, 225; blue, 225 }  ,fill opacity=1 ]   (90,170) -- (370,170) ;
\draw [color={rgb, 255:red, 225; green, 225; blue, 225 }  ,draw opacity=1 ]   (90,50) -- (370,50) ;
\draw [color={rgb, 255:red, 0; green, 0; blue, 0 }  ,draw opacity=1 ]   (230,230) -- (180,260) ;
\draw [shift={(180,260)}, rotate = 149.04] [color={rgb, 255:red, 0; green, 0; blue, 0 }  ,draw opacity=1 ][fill={rgb, 255:red, 0; green, 0; blue, 0 }  ,fill opacity=1 ][line width=0.75]      (0, 0) circle [x radius= 3.35, y radius= 3.35]   ;
\draw [color={rgb, 255:red, 0; green, 0; blue, 0 }  ,draw opacity=1 ]   (200,440) -- (200,470) ;
\draw [shift={(200,470)}, rotate = 90] [color={rgb, 255:red, 0; green, 0; blue, 0 }  ,draw opacity=1 ][fill={rgb, 255:red, 0; green, 0; blue, 0 }  ,fill opacity=1 ][line width=0.75]      (0, 0) circle [x radius= 3.35, y radius= 3.35]   ;
\draw [color={rgb, 255:red, 0; green, 0; blue, 0 }  ,draw opacity=1 ]   (230,350) -- (200,380) ;
\draw [shift={(200,380)}, rotate = 135] [color={rgb, 255:red, 0; green, 0; blue, 0 }  ,draw opacity=1 ][fill={rgb, 255:red, 0; green, 0; blue, 0 }  ,fill opacity=1 ][line width=0.75]      (0, 0) circle [x radius= 3.35, y radius= 3.35]   ;
\draw [color={rgb, 255:red, 0; green, 0; blue, 0 }  ,draw opacity=1 ]   (160,380) -- (160,410) ;
\draw [shift={(160,410)}, rotate = 90] [color={rgb, 255:red, 0; green, 0; blue, 0 }  ,draw opacity=1 ][fill={rgb, 255:red, 0; green, 0; blue, 0 }  ,fill opacity=1 ][line width=0.75]      (0, 0) circle [x radius= 3.35, y radius= 3.35]   ;
\draw [color={rgb, 255:red, 0; green, 0; blue, 0 }  ,draw opacity=1 ]   (160,350) -- (160,380) ;
\draw [shift={(160,380)}, rotate = 90] [color={rgb, 255:red, 0; green, 0; blue, 0 }  ,draw opacity=1 ][fill={rgb, 255:red, 0; green, 0; blue, 0 }  ,fill opacity=1 ][line width=0.75]      (0, 0) circle [x radius= 3.35, y radius= 3.35]   ;
\draw [color={rgb, 255:red, 0; green, 0; blue, 0 }  ,draw opacity=1 ]   (230,320) -- (160,350) ;
\draw [shift={(160,350)}, rotate = 156.8] [color={rgb, 255:red, 0; green, 0; blue, 0 }  ,draw opacity=1 ][fill={rgb, 255:red, 0; green, 0; blue, 0 }  ,fill opacity=1 ][line width=0.75]      (0, 0) circle [x radius= 3.35, y radius= 3.35]   ;
\draw    (260,440) -- (260,470) ;
\draw [shift={(260,470)}, rotate = 90] [color={rgb, 255:red, 0; green, 0; blue, 0 }  ][fill={rgb, 255:red, 0; green, 0; blue, 0 }  ][line width=0.75]      (0, 0) circle [x radius= 3.35, y radius= 3.35]   ;
\draw    (200,410) -- (200,440) ;
\draw [shift={(200,440)}, rotate = 90] [color={rgb, 255:red, 0; green, 0; blue, 0 }  ][fill={rgb, 255:red, 0; green, 0; blue, 0 }  ][line width=0.75]      (0, 0) circle [x radius= 3.35, y radius= 3.35]   ;
\draw    (260,380) -- (260,410) ;
\draw [shift={(260,410)}, rotate = 90] [color={rgb, 255:red, 0; green, 0; blue, 0 }  ][fill={rgb, 255:red, 0; green, 0; blue, 0 }  ][line width=0.75]      (0, 0) circle [x radius= 3.35, y radius= 3.35]   ;
\draw    (200,380) -- (200,410) ;
\draw [shift={(200,410)}, rotate = 90] [color={rgb, 255:red, 0; green, 0; blue, 0 }  ][fill={rgb, 255:red, 0; green, 0; blue, 0 }  ][line width=0.75]      (0, 0) circle [x radius= 3.35, y radius= 3.35]   ;
\draw    (230,320) -- (230,350) ;
\draw [shift={(230,350)}, rotate = 90] [color={rgb, 255:red, 0; green, 0; blue, 0 }  ][fill={rgb, 255:red, 0; green, 0; blue, 0 }  ][line width=0.75]      (0, 0) circle [x radius= 3.35, y radius= 3.35]   ;
\draw    (230,290) -- (230,320) ;
\draw [shift={(230,320)}, rotate = 90] [color={rgb, 255:red, 0; green, 0; blue, 0 }  ][fill={rgb, 255:red, 0; green, 0; blue, 0 }  ][line width=0.75]      (0, 0) circle [x radius= 3.35, y radius= 3.35]   ;
\draw    (230,260) -- (230,290) ;
\draw [shift={(230,290)}, rotate = 90] [color={rgb, 255:red, 0; green, 0; blue, 0 }  ][fill={rgb, 255:red, 0; green, 0; blue, 0 }  ][line width=0.75]      (0, 0) circle [x radius= 3.35, y radius= 3.35]   ;
\draw    (230,230) -- (230,260) ;
\draw [shift={(230,260)}, rotate = 90] [color={rgb, 255:red, 0; green, 0; blue, 0 }  ][fill={rgb, 255:red, 0; green, 0; blue, 0 }  ][line width=0.75]      (0, 0) circle [x radius= 3.35, y radius= 3.35]   ;
\draw    (230,200) -- (230,230) ;
\draw [shift={(230,230)}, rotate = 90] [color={rgb, 255:red, 0; green, 0; blue, 0 }  ][fill={rgb, 255:red, 0; green, 0; blue, 0 }  ][line width=0.75]      (0, 0) circle [x radius= 3.35, y radius= 3.35]   ;
\draw    (230,170) -- (230,200) ;
\draw [shift={(230,200)}, rotate = 90] [color={rgb, 255:red, 0; green, 0; blue, 0 }  ][fill={rgb, 255:red, 0; green, 0; blue, 0 }  ][line width=0.75]      (0, 0) circle [x radius= 3.35, y radius= 3.35]   ;
\draw    (230,140) -- (230,170) ;
\draw [shift={(230,170)}, rotate = 90] [color={rgb, 255:red, 0; green, 0; blue, 0 }  ][fill={rgb, 255:red, 0; green, 0; blue, 0 }  ][line width=0.75]      (0, 0) circle [x radius= 3.35, y radius= 3.35]   ;
\draw    (230,110) -- (230,140) ;
\draw [shift={(230,140)}, rotate = 90] [color={rgb, 255:red, 0; green, 0; blue, 0 }  ][fill={rgb, 255:red, 0; green, 0; blue, 0 }  ][line width=0.75]      (0, 0) circle [x radius= 3.35, y radius= 3.35]   ;
\draw    (230,80) -- (230,110) ;
\draw [shift={(230,110)}, rotate = 90] [color={rgb, 255:red, 0; green, 0; blue, 0 }  ][fill={rgb, 255:red, 0; green, 0; blue, 0 }  ][line width=0.75]      (0, 0) circle [x radius= 3.35, y radius= 3.35]   ;
\draw    (230,50) -- (230,80) ;
\draw [shift={(230,80)}, rotate = 90] [color={rgb, 255:red, 0; green, 0; blue, 0 }  ][fill={rgb, 255:red, 0; green, 0; blue, 0 }  ][line width=0.75]      (0, 0) circle [x radius= 3.35, y radius= 3.35]   ;
\draw    (260,410) -- (260,440) ;
\draw [shift={(260,440)}, rotate = 90] [color={rgb, 255:red, 0; green, 0; blue, 0 }  ][fill={rgb, 255:red, 0; green, 0; blue, 0 }  ][line width=0.75]      (0, 0) circle [x radius= 3.35, y radius= 3.35]   ;
\draw    (230,350) -- (260,380) ;
\draw [shift={(260,380)}, rotate = 45] [color={rgb, 255:red, 0; green, 0; blue, 0 }  ][fill={rgb, 255:red, 0; green, 0; blue, 0 }  ][line width=0.75]      (0, 0) circle [x radius= 3.35, y radius= 3.35]   ;
\draw    (300,380) -- (300,410) ;
\draw [shift={(300,410)}, rotate = 90] [color={rgb, 255:red, 0; green, 0; blue, 0 }  ][fill={rgb, 255:red, 0; green, 0; blue, 0 }  ][line width=0.75]      (0, 0) circle [x radius= 3.35, y radius= 3.35]   ;
\draw    (300,350) -- (300,380) ;
\draw [shift={(300,380)}, rotate = 90] [color={rgb, 255:red, 0; green, 0; blue, 0 }  ][fill={rgb, 255:red, 0; green, 0; blue, 0 }  ][line width=0.75]      (0, 0) circle [x radius= 3.35, y radius= 3.35]   ;
\draw    (230,320) -- (300,350) ;
\draw [shift={(300,350)}, rotate = 23.2] [color={rgb, 255:red, 0; green, 0; blue, 0 }  ][fill={rgb, 255:red, 0; green, 0; blue, 0 }  ][line width=0.75]      (0, 0) circle [x radius= 3.35, y radius= 3.35]   ;
\draw    (230,80) -- (270,110) ;
\draw [shift={(270,110)}, rotate = 36.87] [color={rgb, 255:red, 0; green, 0; blue, 0 }  ][fill={rgb, 255:red, 0; green, 0; blue, 0 }  ][line width=0.75]      (0, 0) circle [x radius= 3.35, y radius= 3.35]   ;
\draw [color={rgb, 255:red, 0; green, 0; blue, 0 }  ,draw opacity=1 ]   (230,80) -- (190,110) ;
\draw [shift={(190,110)}, rotate = 143.13] [color={rgb, 255:red, 0; green, 0; blue, 0 }  ,draw opacity=1 ][fill={rgb, 255:red, 0; green, 0; blue, 0 }  ,fill opacity=1 ][line width=0.75]      (0, 0) circle [x radius= 3.35, y radius= 3.35]   ;
\draw    (180,260) -- (180,290) ;
\draw [shift={(180,290)}, rotate = 90] [color={rgb, 255:red, 0; green, 0; blue, 0 }  ][fill={rgb, 255:red, 0; green, 0; blue, 0 }  ][line width=0.75]      (0, 0) circle [x radius= 3.35, y radius= 3.35]   ;
\draw    (230,30) -- (230,50) ;
\draw [shift={(230,50)}, rotate = 90] [color={rgb, 255:red, 0; green, 0; blue, 0 }  ][fill={rgb, 255:red, 0; green, 0; blue, 0 }  ][line width=0.75]      (0, 0) circle [x radius= 3.35, y radius= 3.35]   ;
\draw  [dash pattern={on 0.84pt off 2.51pt}]  (230,30) -- (230,10) ;
\draw [color={rgb, 255:red, 0; green, 0; blue, 0 }  ,draw opacity=1 ]   (230,230) -- (280,260) ;
\draw [shift={(280,260)}, rotate = 30.96] [color={rgb, 255:red, 0; green, 0; blue, 0 }  ,draw opacity=1 ][fill={rgb, 255:red, 0; green, 0; blue, 0 }  ,fill opacity=1 ][line width=0.75]      (0, 0) circle [x radius= 3.35, y radius= 3.35]   ;
\draw    (280,260) -- (280,290) ;
\draw [shift={(280,290)}, rotate = 90] [color={rgb, 255:red, 0; green, 0; blue, 0 }  ][fill={rgb, 255:red, 0; green, 0; blue, 0 }  ][line width=0.75]      (0, 0) circle [x radius= 3.35, y radius= 3.35]   ;

\draw (67,102.4) node [anchor=north west][inner sep=0.75pt]    {$\textcolor[rgb]{0.88,0.88,0.88}{0}$};
\draw (59,132.4) node [anchor=north west][inner sep=0.75pt]    {$\textcolor[rgb]{0.88,0.88,0.88}{-1}$};
\draw (59,162.4) node [anchor=north west][inner sep=0.75pt]  [color={rgb, 255:red, 225; green, 225; blue, 225 }  ,opacity=1 ]  {$-2$};
\draw (59,192.4) node [anchor=north west][inner sep=0.75pt]  [color={rgb, 255:red, 225; green, 225; blue, 225 }  ,opacity=1 ]  {$-3$};
\draw (59,222.4) node [anchor=north west][inner sep=0.75pt]  [color={rgb, 255:red, 225; green, 225; blue, 225 }  ,opacity=1 ]  {$-4$};
\draw (59,252.4) node [anchor=north west][inner sep=0.75pt]  [color={rgb, 255:red, 225; green, 225; blue, 225 }  ,opacity=1 ]  {$-5$};
\draw (59,282.4) node [anchor=north west][inner sep=0.75pt]  [color={rgb, 255:red, 225; green, 225; blue, 225 }  ,opacity=1 ]  {$-6$};
\draw (59,312.4) node [anchor=north west][inner sep=0.75pt]  [color={rgb, 255:red, 225; green, 225; blue, 225 }  ,opacity=1 ]  {$-7$};
\draw (59,342.4) node [anchor=north west][inner sep=0.75pt]  [color={rgb, 255:red, 225; green, 225; blue, 225 }  ,opacity=1 ]  {$-8$};
\draw (59,372.4) node [anchor=north west][inner sep=0.75pt]  [color={rgb, 255:red, 225; green, 225; blue, 225 }  ,opacity=1 ]  {$-9$};
\draw (51,402.4) node [anchor=north west][inner sep=0.75pt]  [color={rgb, 255:red, 225; green, 225; blue, 225 }  ,opacity=1 ]  {$-10$};
\draw (51,492.4) node [anchor=north west][inner sep=0.75pt]  [color={rgb, 255:red, 225; green, 225; blue, 225 }  ,opacity=1 ]  {$-13$};
\draw (51,462.4) node [anchor=north west][inner sep=0.75pt]  [color={rgb, 255:red, 225; green, 225; blue, 225 }  ,opacity=1 ]  {$-12$};
\draw (51,432.4) node [anchor=north west][inner sep=0.75pt]  [color={rgb, 255:red, 225; green, 225; blue, 225 }  ,opacity=1 ]  {$-11$};
\draw (67,72.4) node [anchor=north west][inner sep=0.75pt]    {$\textcolor[rgb]{0.88,0.88,0.88}{1}$};
\draw (67,42.4) node [anchor=north west][inner sep=0.75pt]    {$\textcolor[rgb]{0.88,0.88,0.88}{2}$};

\end{tikzpicture}}

\end{center}
\noindent For a conceptual reason for why distinct plane curve singularities may have isomorphic lattice cohomology modules see \cite{NS25}.

\begin{remark}
    The latter example implies, in particular, that $\bH^*$ is not a complete topological invariant in the case of multiple branches, compare with Theorem \ref{thm:MAIN}. See \cite{FiltLC} for the filtered lattice homology of curve singularities which \emph{is} a complete invariant in the reducible case as well.
\end{remark}

\section{Lattice cohomology and the multivariate Poincaré series}\label{sec:poincaré}

\noindent The Poincaré series $P_{C,o}$ (see Remark \ref{rem:poincaré}) of a plane curve singularity $(C,o)$ is equivalent to the Alexander polynomial $\Delta_{L_C}$ of its link $L_C$ \cite{CDG03}
and hence to its embedded topological type \cite{Y84}. Thus, the lattice cohomology of a plane curve singularity is determined by its Poincaré series. In this section we give an example showing that this is no longer the case in higher codimensions.\vspace{3
mm}

 In \cite{CDG7}, Campillo, Delgado and Gusein-Zade give the following example: Let $(C,o) = (C_1,o)\cup(C_2,o) \subset (\bC^5,o)$ and $(C',o) = (C'_1,o)\cup(C'_2,o)\subset (\bC^6,o)$ be the germs given by the parametrizations $C_1 = \{(t^2,t^3,t^2,t^4,t^5) \}$, $C_2 = \{(u^2,u^3,u^2,u^4,u^6) \}$, $C'_1 =  \{(t^3,t^4,t^5,t^4,t^5,t^6 )\}$ and $C'_2=\{(u^3,u^4,u^5,u^5,u^6,u^7 ) \}$. On one hand, the Poincaré series is $P(t_1,t_2) = 1+t_1^3t_2^3$ in both cases. On the other hand, their lattice cohomologies are different, see the weight tables below.\vspace{3mm}

\begin{center}

\begin{tabular}{|ccccc}
     2 & 1 & 0 & 1 & 0  \\
    1 & 0 & -1 & 0 & 1  \\
    0 & -1 & -2 & -1 & 0 \\
    1 & 0 & -1 & 0 & 1  \\
    0 & 1 & 0 & 1 & 2  \\
    \hline
\end{tabular} \hspace{20mm} \begin{tabular}{|ccccc}
     0 & -1 & -2 & -3 & -4 \\
    -1 & -2 & -3 & -4 & -3 \\
    0 & -1 & -2 & -3 & -2 \\
    1 & 0 & -1 & -2 & -1 \\
    0 & 1 & 0 & -1 & 0 \\
    \hline
\end{tabular}\vspace{4mm}

\footnotesize

Weight functions of the curve singularities $(C,o)$ and $(C',o)$ respectively.\vspace{2mm}
\end{center}

\noindent The weight functions are only shown in a finite rectangle (between the origin and the conductors). This is enough to determine $\bH^*$, see \cite[Lemma 4.2.1]{AgNeCurve}.

The quantity $\min(w_0)$ is an invariant of the bigrading on $\bH^*$: it is the minimal integer $n$ such than $\bH^0_{2n} \not= 0$. The tables show that $\min(w_0) = -2$ for $(C,o)$, whereas $\min(w_0) = -4$ in the case of $(C',o)$. It follows that $\bH^*(C,o) \not\cong \bH^*(C',o)$, as claimed.


\end{document}